\documentclass[12pt]{amsart}
\usepackage{amssymb,amsfonts,amsmath,amsopn,amstext,amscd,color,
latexsym,mathrsfs,stackrel,url,verbatim}
\usepackage[all]{xy}
\usepackage{mathabx}
\usepackage{hyperref}
\usepackage[utf8]{inputenc}

\input xy
\xyoption{all}

\usepackage[active]{srcltx}

\theoremstyle{remark}

\theoremstyle{definition}

\newtheorem{dfn}[subsubsection]{Definition}

\theoremstyle{plain}

\newtheorem{thm}[subsubsection]{Theorem}

\newtheorem{lemma}[subsubsection]{Lemma}
\newtheorem{cor}[subsubsection]{Corollary}
\newtheorem{prop}[subsubsection]{Proposition}

\newcommand{\vpi}{{\varpi}}

\newcommand{\bbN}{{\mathbb N}}

\newcommand{\bbQ}{{\mathbb Q}}
\newcommand{\bbR}{{\mathbb R}}

\newcommand{\bbZ}{{\mathbb Z}}

\newcommand{\frb}{{\mathfrak b}}

\newcommand{\frg}{{\mathfrak g}}

\newcommand{\fro}{{\mathfrak o}}

\newcommand{\frt}{{\mathfrak t}}

\newcommand{\frS}{{\mathfrak S}}

\newcommand{\frU}{{\mathfrak U}}
\newcommand{\frV}{{\mathfrak V}}

\newcommand{\frX}{{\mathfrak X}}
\newcommand{\frY}{{\mathfrak Y}}
\newcommand{\frZ}{{\mathfrak Z}}

\newcommand{\cB}{{\mathcal B}}

\newcommand{\cF}{{\mathcal F}}
\newcommand{\cG}{{\mathcal G}}
\newcommand{\cH}{{\mathcal H}}
\newcommand{\cI}{{\mathscr I}}
\newcommand{\cJ}{{\mathscr J}}

\newcommand{\cO}{{\mathscr O}}
\newcommand{\cP}{{\mathcal P}}

\newcommand{\cT}{{\mathcal T}}

\newcommand{\sA}{{\mathscr A}}

\newcommand{\sD}{{\mathscr D}}

\newcommand{\sI}{{\mathscr I}}

\newcommand{\sK}{{\mathscr K}}
\newcommand{\sL}{{\mathscr L}}
\newcommand{\sM}{{\mathscr M}}
\newcommand{\sN}{{\mathscr N}}
\newcommand{\sO}{{\mathscr O}}

\newcommand{\sV}{{\mathscr V}}

\newcommand{\Q}{{\mathbb Q}}

\newcommand{\der}{\partial}
\newcommand{\lan}{\langle}
\newcommand{\ran}{\rangle}

\newcommand{\Hom}{{\rm Hom}}

\newcommand{\Lie}{{\rm{Lie}}}

\newcommand{\Spec}{{\rm Spec\,}}
\newcommand{\Spf}{{\rm Spf}}

\newcommand{\coker}{{\rm coker}}

\newcommand{\im}{{\rm im}}

\newcommand{\ord}{{\rm ord}}

\newcommand{\hsD}{{\widehat{\sD}}} 

\newcommand{\car}{\stackrel{\simeq}{\longrightarrow}}

\newcommand{\varep}{\varepsilon}

\newcommand{\rig}{\rightarrow}
\newcommand{\trig}{\twoheadrightarrow}

\newcommand{\hrig}{\hookrightarrow}
\newcommand{\GG}{{\mathcal G}}

\newcommand{\uk}{\underline{k}}

\newcommand{\ut}{\underline{t}}

\newcommand{\uq}{\underline{q}}

\newcommand{\uder}{\underline{\partial}}

\newcommand{\unu}{\underline{\nu}}

\begin{document}
\title{Kashiwara's theorem for twisted arithmetic differential operators}
\author{Christine Huyghe}
\address{IRMA, Universit\'e de Strasbourg, 7 rue Ren\'e Descartes, 67084 Strasbourg cedex, France}
\email{huyghe@math.unistra.fr}
\author{Tobias Schmidt}
\address{IRMAR, Universit\'e de Rennes 1, Campus Beaulieu, 35042 Rennes cedex, France}
\email{Tobias.Schmidt@univ-rennes1.fr}

\thanks{}
\begin{abstract} We establish a version of Kashiwara's theorem for twisted sheaves of Berthelot's arithmetic differential operators
for a closed immersion between smooth $p$-adic formal schemes. As an application, we give a geometric construction of simple modules
for crystalline distribution algebras of reductive groups. 
 \end{abstract}

\maketitle

\tableofcontents

\section{Introduction}

Let $X$ be a smooth complex variety and $Y\subset X$ a smooth closed subvariety. 
A basic result in $\sD$-module theory is Kashiwara's theorem which states an equivalence of categories
between the category of $\sD_Y$-modules, quasi-coherent over $\sO_Y$, and the category of $\sD_X$-modules, quasi-coherent over $\sO_X$, with support in $Y$. 
Recall that a {\it twisted sheaf of differential operators} on $X$ is a 
sheaf of rings $\sA$ on $X$ together with a ring homomorphism $\iota: \sO_{X}\rightarrow \sA $ 
such that the 
pair $(\iota,\sA )$ is locally isomorphic to the pair $(\text{can},\sD_X)$ where $\text{can}: \sO_{X}\rightarrow \sD_X$ is the canonical inclusion. Originally, such twisted sheaves were introduced in the early 1980s by Beilinson-Bernstein in order to study localisations of Lie algebra representations with general infinitesimal character on complex flag varieties \cite{BB81}. It is well-known that the right module version of Kashiwara's theorem generalizes to twisted sheaves of differential operators \cite[4.3]{Milicic93}. Under additional hypotheses on the twisted sheaf $\sA$ (e.g. an analogue of the order filtration on $\sA$) one may establish side-changing functors in the general setting of $\sA$-modules and then deduce a version of Kashiwara's theorem for left modules \cite[2.3]{Kashiwara_rims622}. 

\vskip5pt 

In the arithmetic setting, let $\fro$ denote a complete discrete valuation ring of mixed characteristic $(0,p)$ with uniformizer $\varpi$, field of fractions $K$ and perfect residue field. Let $\frX$ be a smooth formal scheme over $\fro$ and 
let $\sD^{\dagger}_{\frX,\bbQ}$ be the sheaf of arithmetic differential operators on $\frX$ \cite{BerthelotDI}. 
If $\frY\subset \frX$ is a closed smooth formal subscheme, Berthelot's version of Kashiwara's theorem gives an equivalence 
between the category of coherent left $\sD^{\dagger}_{\frY}$-modules and the category of coherent left $\sD^{\dagger}_{\frX}$-modules with support in $\frY$. Berthelot gives a proof of the theorem in his course on arithmetic $\sD$-modules 1997 at the Centre Emile Borel, which, however, is not included in the corresponding reference \cite{BerthelotIntro}. In the appendix of \cite{CaroSysInd} Caro establishes a version of the theorem for log structures and coefficients. 

\vskip5pt

Our goal in this paper is to establish a version of Kashiwara's theorem for twisted sheaves of arithmetic differential operators.
Similar to the complex analytic setting, such twisted sheaves appear naturally in the context of the localization theory of crystalline distribution algebras of reductive groups, when varying the infinitesimal character of representations \cite{HS2,SA1}. Following Beilinson-Bernstein, we define a {\it twisted sheaf of arithmetic differential operators} to be a sheaf of rings $\sA$ on $\frX$ together with a ring homomorphism $\iota: \sO_{\frX,\bbQ }\rightarrow \sA $ 
such that the pair $(\iota,\sA )$ is locally isomorphic to the pair $(\text{can},\sD^{\dagger}_{\frX,\bbQ})$ where 
$\text{can}: \sO_{\frX,\bbQ}\rightarrow \sD^{\dagger}_{\frX,\bbQ}$ is the canonical inclusion. At this level of generality, as we have explained above, there are no side-changing functors and one may only hope for a right module version of Kashiwara's theorem. 
\vskip5pt 
To formulate our main result, let $\sA$ be a twisted sheaf of differential operators on $\frX$ and let $i: \frY\rightarrow\frX$ be the inclusion of a closed smooth formal subscheme
defined by the ideal $\sI \subset \sO_{\frX}$. Let $\sN_{\sA }(\sI \sA)$ be the normalizer of the right ideal generated by $\sI$ in $\sA$. We show that 
$$\sA_{\frY}:=i^{-1} \big(\sN_\sA(\sI \sA)/\sI \sA\big)\hskip10pt \text{resp.}\hskip10pt \sA_{\frY\rightarrow\frX}:=i^{*} (\sA)$$ are a twisted sheaf of arithmetic differential operators on $\frY$ resp. a $(\sA_{\frY},i^{-1}\sA)$-bimodule. We obtain an adjoint pair of functors 
$( i_{\sA,+}, i_\sA^\natural)$ between the categories of right modules over $\sA_{\frY}$ and $\sA$ respectively: the direct image 

$$  i_{\sA,+}\sN:= i_* (\sN \otimes_{{\sA}_{\frY}}\sA_{\frY\rightarrow\frX} )$$
from right $\sA_{\frY}$-modules to right $\sA$-modules
and, in the opposite direction, the restriction functor 
$$i_\sA^\natural\sM:=
{\mathcal Hom}_{i^{-1}\sA}(\sA_{\frY\rightarrow\frX},i^{-1}\sM).$$

Let ${\rm Coh}^r(\sA_{\frY})$ and ${\rm Coh}^{r,\frY}(\sA)$ be the categories of coherent 
 right $\sA_{\frY}$-modules and coherent 
 right $\sA$-modules with support in $\frY$, respectively. Our main result is the following.

\vskip5pt 

{\bf Theorem} (cf. \ref {thm-BK_A}). The functors $i_{\sA,+},i_\sA^\natural$ induces 
mutually inverse equivalences of categories
$$
\xymatrix{
 {\rm Coh}^{r}(\sA_\frY) \;\;  \ar@<1ex>[r]^{i_{\sA,+}} &\;\; {\rm Coh}^{r,\frY}(\sA)  \ar@<1ex>[l]^{i_\sA^\natural}_{\simeq}
}.
$$
For the proof, we proceed as follows. We first establish some complements on arithmetic differential operators, notably the normalizer description for operators on closed subspaces. We then
give a full and self-contained proof of the Berthelot-Kashiwara theorem for 
left $\sD^{\dagger}_{\frX,\bbQ}$-modules. Note that
Caro's logarithmic Kashiwara theorem for coefficients \cite{CaroSysInd} contains this result as a special case. However, this special case is buried under a lot of additional notation. We therefore believe that it is instructive, and a useful basis for our future work, to have a clear direct proof in this special case, using only the tools of the basic reference \cite{BerthelotDI}. As in \cite{CaroSysInd}, the key ingredient is a lemma of Berthelot on a certain matrix identity involving matrices over arithmetic differential operators. We give a full proof of this lemma in \ref{section_keylemma} (in \cite{CaroSysInd} only the rank $1$-case is really considered). We then use side-changing functors to obtain Berthelot-Kashiwara for right $\sD^{\dagger}_{\frX,\bbQ}$-modules. Finally, we prove sufficiently many properties and compatibilities of the adjoint pair $(i_{\sA,+},i_\sA^\natural)$ to reduce the proof of the main theorem to a local situation. This allows us to undo the twist and then conclude via right Berthelot-Kashiwara. 

\vskip5pt 

In the final section, we give an application to the representation theory of crystalline distributions algebras. 
We fix a connected split reductive group $G$ over $\fro$ and let $D^{\dagger}(\GG)_{\bbQ}$ be its crystalline distribution algebra, as introduced and studied in \cite{HS1}. Irreducible modules over $D^{\dagger}(\GG)_{\bbQ}$ can be considered as local data for certain admissible locally analytic $G(K)$-representations and thus, are of interest in the so-called $p$-adic local Langlands programme. 
We let $\frX$ be the formal flag variety of $G$. In \cite{SA1} Sarrazola-Alzate generalizes a classical construction of Beilinson-Bernstein \cite{BB81} and Borho-Brylinski \cite{BoBr89} to the arithmetic setting and constructs a family of twisted sheaves of arithmetic differential operators $\sD^{\dagger}_{\frX,\bbQ,\lambda}$ on $\frX$, indexed by certain characters $\lambda$ of a Cartan subalgebra of $Lie(G)\otimes\bbQ$. We apply our Berthelot-Kashiwara theorem to the sheaves $\sD^{\dagger}_{\frX,\bbQ,\lambda}$. For algebraic $\lambda$ (i.e. when $\lambda$ lifts to a character of a maximal split torus in $G$) this leads to a class of simple $\sD^{\dagger}_{\frX,\bbQ,\lambda}$-modules, parametrized by closed smooth subschemes $\frY\subset \frX$.
By the arithmetic localisation theorem \cite{HS2}, their global sections give rise to simple $D^{\dagger}(\GG)_{\bbQ}$-modules.
 
 \vskip5pt
 
 {\it Acknowledgements.} We thank Daniel Caro for having made available to us a very preliminary version of the article \cite{CaroSysInd} and for helpful discussions regarding certain points in this article. 
  
 \vskip5pt 
 
{\it Notations and Conventions.} \label{notations}
Throughout the article, $\fro$ denotes a complete discrete valuation ring with mixed characteristic $(0,p)$.
We denote by $K$ its fraction field and by $k$ its residue field, which is assumed to be perfect. 
We let $\vpi$ be a uniformizer of $\fro$. A formal scheme $\frX$ over $\fro$ which is locally noetherian and such that $\vpi \cO_{\frX}$ is an ideal of definition is called an {\it $\fro$-formal scheme}. We denote its special fibre by $\frX_s$.

\section{Complements on arithmetic differential operators}

\subsection{Arithmetic differential operators} 
In this subsection we assume a certain familiarity with the basic notions of divided powers and divided power envelopes. Our basic references are \cite{BerthelotDI,Berthelot_Dmod2}

\vskip5pt 

Let $\frX$ be an $\fro$-formal scheme, which is smooth over $\frS:=\Spf(\fro)$, with structure sheaf $\cO_{\frX}$.

\vskip5pt

Let $m\geq 0$. Let $\cP^n_{\frX/\frS,(m)}, n\geq 0$ be the projective system of sheaves of principal parts of level $m$ and order $n$ of $\frX$ relative to $\frS$. 
There are two morphisms $p_0,p_1: \cP^n_{\frX/\frS,(m)}\rightarrow \cO_{\frX}$, induced from the two projection morphisms 
$\frX\times\frX\rightarrow \frX$, 
making $\cP^n_{\frX/\frS,(m)}$ a commutative $\cO_{\frX}$-algebra in two ways, on the left (via $p_0$) and on the right (via $p_1$). The two structures are denoted by $d_i : \cO_{\frX}\rightarrow \cP^n_{\frX/\frS,(m)}$ for $i=0,1$. If $\frX$ has \'etale coordinates $t_1,...,t_d$ and $\tau_i:=p_1^*(t_i)-p_0^*(t_i)\in \cO_{\frX \times \frX}$, then $\cP^n_{\frX/\frS,(m)}$ is a free $\cO_{\frX}$-module (for both its left and right structure) on the basis $\underline{\tau}^{ \{ \underline{k}\} }:=\tau_1^{\{k_1\}}\cdot\cdot\cdot \tau_d^{\{k_d\}}$ for $| \underline{k} | \leq n$.

\vskip5pt

The sheaf of arithmetic differential operators on $\frX$ of level $m$ and order $n$ is the $\cO_{\frX}$-dual $\sD_{\frX,n}^{(m)}: = \mathcal{H}om_{\cO_{\frX}}(\cP^n_{\frX/\frS,(m)},\cO_{\frX})$. An element $P\in \sD_{\frX,n}^{(m)}$ acts on $\cO_{\frX}$ via the composition 
$$ \cO_{\frX}\stackrel{d_1}{\longrightarrow} \cP^n_{\frX/\frS,(m)} \stackrel{P}{\longrightarrow} \cO_{\frX}.$$

 The union $\sD_{\frX}^{(m)}:=\cup_n \sD_{\frX,n}^{(m)}$ is a ring and there is a natural ring homomorphism $\sD_{\frX}^{(m)}\rightarrow \sD_{\frX}^{(m+1)}$. We denote by $\hsD_{\frX}^{(m)}=\varprojlim_i \sD^{(m)}_{\frX}/\varpi^{i}$ the $\varpi$-adic completion and let $$\sD^{\dagger}_{\frX}:=\varinjlim_m \hsD_{\frX}^{(m)}\hskip10pt \text{and}\hskip10pt \sD^{\dagger}_{\frX,\bbQ}:=\sD^{\dagger}_{\frX}\otimes_{\bbZ}\bbQ.$$

 \vskip5pt
 
 We shall also need to consider the usual (i.e. with divided powers) 
ring of algebraic differential operators $\sD_{\frX}:=\varinjlim_m \sD_{\frX}^{(m)}$ on the $\fro$-formal scheme $\frX$, cf. \cite[16.8]{EGA_IV_4}. We denote by $\hsD_{\frX}$ its $\varpi$-adic completion. 

\vskip5pt 
It will be useful to make the following definition.
  
 \begin{dfn} An {\it ${\cO_\frX}$-ring} is a sheaf of rings $\sA$ on $\frX$, together with a ring morphism $\cO_{\frX}\rightarrow\sA$. 
\end{dfn}
There is the obvious notion of a {\it morphism} between two $\cO_{\frX}$-rings and one obtains thereby a category of ${\cO_\frX}$-rings.
There are obvious variants of this category when the structure sheaf $\cO_{\frX}$ is replaced by another sheaf associated with the formal scheme $\frX$, such as $\cO_{\frX,\bbQ}$ or $\cO_{\frX,\bbQ}/\cJ$ (for an ideal sheaf $\cJ\subseteq \cO_{\frX,\bbQ}$).

\vskip5pt 
 
 All the rings $\sD_{\frX}^{(m)},  \sD_{\frX}, \hsD_{\frX}^{(m)},  \hsD_{\frX}, \sD^{\dagger}_{\frX}$ are $\cO_{\frX}$-rings, and $\sD^{\dagger}_{\frX,\bbQ}$ is even an $\cO_{\frX,\bbQ}$-ring.

 \subsection{Side-changing functors}\label{side_changing}
 
 Let $\frX$ be an $\fro$-formal scheme of relative dimension $d$, which is smooth over $\frS:=\Spf(\fro)$.
 Let $$\omega_\frX:=\bigwedge^{d} \Omega^1_{\frX / \frS}$$ be the module of differentials of highest degree, with its 
 natural right $\sD_{\frX}$-action \cite[1.2.1]{Berthelot_Dmod2}. As in the classical setting, the functors $\omega_{\frX}\otimes_{\cO_{\frX}}(-)$ and 
  $\cH om_{\cO_{\frX}}(\omega_{\frX},-)$ induce mutually inverse equivalences of categories between left and right $\sD^{(m)}_{\frX}$-modules, for any $m\geq 0$ \cite[1.2.7(c)]{Berthelot_Dmod2}. The following proposition is certainly well-known, we record it for a lack of reference.
  
  \begin{prop}\label{supp_condition} Let $Y \subseteq \frX_s$ be a closed subset. The functors $\omega_{\frX}\otimes_{\cO_{\frX}}(-)$ and 
  $\cH om_{\cO_{\frX}}(\omega_{\frX},-)$ induce mutually inverse equivalences of categories between 
  left and right $\sD^{\dagger}_{\frX,\Q}$-modules supported on $Y$. 
  \end{prop}
\begin{proof} Since $\omega_\frX$ is a coherent $\cO_{\frX}$-module, its spaces of local sections over affine opens in $\frX$ are $\varpi$-adically complete. Its $\sD_{\frX}$-action extends therefore to a $\hsD_{\frX}$-action. The functors $\omega_{\frX}\otimes_{\cO_{\frX}}(-)$ and 
  $\cH om_{\cO_{\frX}}(\omega_{\frX},-)$ descend therefore to equivalences between left and right $\hsD^{(m)}_{\frX}$-modules.
 Inverting $p$ and passing to the inductive limit over all $m\geq 0$ yields the proposition in the case $Y=\frX_s$. The general case follows from the fact that both functors $\omega_{\frX}\otimes_{\cO_{\frX}}(-)$ and 
  $\cH om_{\cO_{\frX}}(\omega_{\frX},-)=(-)\otimes_{\cO_{\frX}}\omega_{\frX}^{-1}$ preserve the support condition. 
\end{proof}

\subsection{Operators on closed subspaces} 
Let 
$$i : \frY \longrightarrow \frX$$ be a closed immersion between two  $\fro$-formal schemes, which are smooth over $\frS:=\Spf(\fro)$. 
Let $r:=\dim \frY$ and $d:=\dim \frX$. 

\vskip5pt 

It is well-known that the adjoint pair of functors $(i_*,i^{-1})$ induces an equivalence of categories between abelian sheaves on $\frY$ and abelian sheaves on $\frX$ with support in $\frY$. We denote by $i^*$ the functor $\cO_{\frY}\otimes_{i^{-1}\cO_{\frX}}i^{-1}(-)$ from $\cO_{\frX}$-modules to 
$\cO_{\frY}$-modules.

\vskip5pt



Let $\cI\subseteq \cO_{\frX}$ be the ideal sheaf defining the closed immersion $i$. There is the Hom sheaf
$\mathcal{H}om_{\fro} (i^{-1}\cI,\cO_{\frY}) $ on $\frY$, which is an $\cO_{\frY}$-module via multiplication on the target, i.e. 
$(sf) (a):= s(f(a))$ for local sections $s\in \cO_{\frY}$,  $f\in \mathcal{H}om_{\fro} (i^{-1}\cI,\cO_{\frY}) $, 
and $a\in i^{-1}\cI$. Similarly, the sheaf $\mathcal{H}om_{K} (i^{-1}\cI_\bbQ,\cO_{\frY,\bbQ}) $ is an $\cO_{\frY,\bbQ}$-module.

\begin{lemma}\label{lem-theta} Let $\sA\in \{ \sD^{(m)}_{\frX}, \hsD^{(m)}_{\frX}, \sD^{\dagger}_{\frX} \}$. The restriction map $\sA\rightarrow \mathcal{H}om_{\fro} (\cI,\cO_{\frX}),
P\mapsto P | _\cI$ induces a $\cO_{\frY}$-linear morphism 
 $$i^*\sA \longrightarrow \mathcal{H}om_{\fro} (i^{-1}\cI,\cO_{\frY}).$$ 

\end{lemma} 
\begin{proof}
The morphism $i^{\sharp}: i^{-1}\cO_{\frX}\rightarrow \cO_{\frY}$ induces a morphism
$$\mathcal{H}om_{\fro} (i^{-1}\cI,i^{-1}\cO_{\frX})\longrightarrow \mathcal{H}om_{\fro} (i^{-1}\cI,\cO_{\frY}) .$$
Let $\sA\in \{ \sD^{(m)}_{\frX}, \hsD^{(m)}_{\frX}, \sD^{\dagger}_{\frX} \}$.
Applying $i^{-1}$ to the restriction morphism $\sA\rightarrow \mathcal{H}om_{\fro} (\cI,\cO_{\frX})$ and composing with the above morphism yields a morphism 
$i^{-1} \sA\rightarrow \mathcal{H}om_{\fro} (i^{-1}\cI,\cO_{\frY}).$ Since the latter is $i^{-1}\cO_{\frX}$-linear and the target a $\cO_{\frY}$-module, it extends to a $\cO_{\frY}$-linear morphism 
$i^*\sA \rightarrow\mathcal{H}om_{\fro} (i^{-1}\cI,\cO_{\frY}),$ as claimed.
\end{proof}
We call the morphism appearing in the lemma $\theta^{(m)},\widehat{\theta}^{(m)},\theta^\dagger$ in the cases $\sD^{(m)}_{\frX}, \hsD^{(m)}_{\frX}, \sD^{\dagger}_{\frX}$ respectively.

\vskip5pt

According to \cite[2.1.4.3]{BerthelotDI}, for any $n \geq 0$, there is a canonical $\cO_{\frY}$-linear morphism 
$$\cP^n_{(m)}(i): i^*\cP^n_{\frX/\frS,(m)}\longrightarrow \cP^n_{\frY/\frS,(m)}.$$ Dualising and taking the union over all $n$ yields a $\cO_\frY$-linear morphism $ \sD_{\frY}^{(m)}\rightarrow i^*\sD_{\frX}^{(m)} .$ Completing $\varpi$-adically, taking the union over all $m\geq 0$ and finally inverting $p$ yields a $\cO_{\frY,\bbQ}$-linear morphism 
$\sD_{\frY,\bbQ}^{\dagger}\longrightarrow i^*\sD_{\frX,\bbQ}^{\dagger} .$

\begin{prop}\label{prop_closed} There is an exact sequence of $\cO_{\frY,\bbQ}$-modules

$$ 0 \longrightarrow \sD_{\frY,\bbQ}^{\dagger}\longrightarrow i^*\sD_{\frX,\bbQ}^{\dagger} \stackrel{\theta_\bbQ^\dagger}{\longrightarrow}\mathcal{H}om_{K} (i^{-1}\cI_\bbQ,\cO_{\frY,\bbQ}). $$

\end{prop}

\begin{proof}  

The exactness is a local question and we may assume that $\frX$ is affine. We let  $A:=\cO(\frX)$ and $I:=\cO(\cI)$. By the Jacobi criterion, e.g. \cite[Prop.3.5]{BoschLuetkeIII}, we may even assume that $\frX$ has \'etale coordinates $t_1,...,t_d\in A$ such that
\begin{itemize}
\item  the images of $t_{1},...,t_{r}$ in $\overline{A}=A/I$ are \'etale coordinates for $\frY$,
\item  the ideal $I\subseteq A$ is generated by $t_{r+1},...,t_d$.
\end{itemize}
Fix $m\geq 0$ and let $P^n_{A/\fro,(m)}=\cO(\cP^n_{\frX/\frS,(m)})$ and $\sD^{(m)}_A=\cO( \sD_{\frX}^{(m)})$.
There is a natural morphism of left $\overline{A}$-modules  
$$ 0 \longrightarrow \ker(\delta) \longrightarrow   \overline{A}\otimes_A P^n_{A/\fro,(m)} \stackrel{\delta}{\longrightarrow} P^n_{\overline{A}/\fro,(m)}\longrightarrow 0$$
where the map $\delta$ equals the global sections of the morphism $i^*\cP^n_{\frX/\frS,(m)}\rightarrow \cP^n_{\frY/\frS,(m)}.$ Since $\delta$ is $A$-linear on the right, we have  $\overline{A}\otimes_A (P^n_{A/\fro,(m)} I) \subseteq \ker(\delta)$. The $\overline{A}$-linear morphism $\delta$ maps
the elements $1\otimes (  \tau_1^{\{k_1\}}\cdot\cdot\cdot \tau_r^{\{k_r\}})$ with $k_1+\cdot\cdot\cdot + k_r \leq n$ bijectively to a $\overline{A}$-basis of $P^n_{\overline{A}/\fro,(m)}$ and so is indeed surjective. Moreover, we have 
$$ \ker(\delta) = \{ \sum_{\underline{k}\in\bbN^d\setminus \bbN^r} \overline{A} \underline{\tau}^{ \{ \underline{k}\} } \}, $$
where $\underline{k}\notin\bbN^r$ means that $k_j >0$ for some $j>r$. In particular, all $\overline{A}$-modules in the above short exact sequence are free and dualizing gives the exact sequence 

$$ 0 \longrightarrow \Hom_{\overline{A}} (P^n_{\overline{A}/\fro,(m)}, \overline{A} ) \longrightarrow   \Hom_{\overline{A}}( \overline{A}\otimes_A P^n_{A/\fro,(m)},\overline{A})  \stackrel{\delta^*}{\longrightarrow}  \Hom_{\overline{A}} (\ker(\delta), \overline{A} ) \longrightarrow 0.$$
Since $\overline{A}$ is flat over $\bbZ$, the localization map $\Hom_{\overline{A}} (\ker(\delta), \overline{A} )\rightarrow \Hom_{\overline{A}_\bbQ} (\ker(\delta)_\bbQ, \overline{A}_\bbQ )$ is injective and we have the exact sequence 

$$ 0 \longrightarrow \Hom_{\overline{A}} (P^n_{\overline{A}/\fro,(m)}, \overline{A} ) \longrightarrow   \Hom_{\overline{A}}( \overline{A}\otimes_A P^n_{A/\fro,(m)},\overline{A}) \stackrel{\delta^*}{\longrightarrow}  \Hom_{\overline{A}_\bbQ} (\ker(\delta)_\bbQ, \overline{A}_\bbQ ). $$
The formula $q! \tau^{ \{ k\} }=\tau^k$, where $q$ is the euclidian division of $k$ by $p^m$ \cite[1.3.5.2]{BerthelotDI}, shows that 
$$ \ker(\delta)_{\bbQ} = \{ \sum_{\underline{k}\in \bbN^d\setminus\bbN^r} \overline{A}_\bbQ \underline{\tau}^{  \underline{k}} \}. $$
On the other hand, given $t_i\in I$ (i.e. $i>r$), the image of $\tau_i^k= (1\otimes t_i - t_i \otimes 1)^k\in A\otimes_{\fro}A$ in the quotient $\overline{A}\otimes_\fro A $ equals $\overline{1}\otimes t_i^k$. Hence, the image of the set $\{ d_1(f) | f\in I \}\subset P^n_{A/\fro,(m)}$ in $\overline{A}_\bbQ \otimes_A P^n_{A/\fro,(m)}$ generates the $\overline{A}_\bbQ$-module $\ker(\delta)_\bbQ$. For a given element 
$$\overline{P}= P + I \sD_{A,n}^{(m)}\in  \Hom_{\overline{A}}( \overline{A}\otimes_A P^n_{A/\fro,(m)},\overline{A}) = 
\overline{A}\otimes_A \Hom_{A}(P^n_{A/\fro,(m)},A) =
\sD_{A,n}^{(m)}/ I  \sD_{A,n}^{(m)}$$ we therefore have 
$$\delta^*(\overline{P})=0 \hskip10pt \text{if and only if}\hskip10pt P \circ d_1 (I) \subseteq I_\bbQ.$$ 
Since $A/I$ is $p$-torsionfree, one has $A \cap I_\bbQ= I$ and so the condition $P \circ d_1 (I) \subseteq I_\bbQ$ is equivalent to 
$P\circ d_1 (I)\subset  A \cap I_\bbQ=I$. In other words,
the sequence 
$$ 0 \longrightarrow \Hom_{\overline{A}} (P^n_{\overline{A}/\fro,(m)}, \overline{A} ) \longrightarrow   \Hom_{\overline{A}}( \overline{A}\otimes_A P^n_{A/\fro,(m)},\overline{A}) \longrightarrow  \Hom_{\fro} (I, \overline{A} ) $$
is exact. Taking the union over $n\geq 0$ yields the exact sequence 
$$ 0 \longrightarrow \sD^{(m)}_{\overline{A}} \longrightarrow  \sD^{(m)}_{A} / I \sD^{(m)}_{A} \stackrel{\theta^{(m)}}{\longrightarrow}  \Hom_{\fro} (I, \overline{A} ),$$ 
where $\theta^{(m)}$ is the global sections of the map appearing in \ref{lem-theta} for $\sA=\sD_\frX^{(m)}$.
By left-exactness of $\varpi$-adic completion, we obtain the exact sequence 
$$ 0 \longrightarrow \widehat{ \sD^{(m)}_{\overline{A}} }\longrightarrow  \widehat{ \sD^{(m)}_{A} / I \sD^{(m)}_{A} }\stackrel{\widehat{\theta^{(m)}}}{\longrightarrow}  \widehat{ \Hom_{\fro} (I, \overline{A} )} .$$

Since $\varpi$-adic completion is exact on finitely generated modules over the noetherian ring 
$\sD^{(m)}_{A}$, the completion of $\sD^{(m)}_{A} / I \sD^{(m)}_{A}$ equals $\hsD^{(m)}_{A} / I \hsD^{(m)}_{A}$. Therefore, 
$\widehat{\theta^{(m)}}$ admits the factorisation
$$\hsD^{(m)}_{A} / I \hsD^{(m)}_{A}\stackrel{\hat{\theta}^{(m)}}\longrightarrow  \Hom_{\fro} (I, \overline{A} )\longrightarrow  \widehat{ \Hom_{\fro} (I, \overline{A} )} $$
where $\widehat{\theta}^{(m)}$ is the global sections of the map appearing in \ref{lem-theta} for $\sA=\hsD_\frX^{(m)}$.
Since $\overline{A}$ is $\varpi$-adically complete, the $\fro$-module $\Hom_{\fro} (I, \overline{A} )$ is $\varpi$-adically separated and hence injects into its $\varpi$-adic completion. Therefore the induced sequence 

$$ 0 \longrightarrow \hsD^{(m)}_{\overline{A}} \longrightarrow  \hsD^{(m)}_{A} / I \hsD^{(m)}_{A} \stackrel{\hat{\theta}^{(m)}}{\longrightarrow}  \Hom_{\fro} (I, \overline{A} ) $$
remains exact. Taking an inductive limit over $m\geq 0$ and inverting $p$ yields the claim. 
\end{proof}
\begin{dfn} Let $\cJ\subseteq\cO_{\frX,\bbQ}$ be an ideal. The {\it normalizer} of the ideal $\cJ \sD_{\frX,\bbQ }^{\dagger}$ is the subset of $\sD_{\frX,\bbQ }^{\dagger}$ equal to 
$$\sN(\cJ \sD_{\frX,\bbQ }^{\dagger}):=\sN_{\sD_{\frX,\bbQ }^{\dagger}}(\cJ \sD_{\frX,\bbQ }^{\dagger}):=\{ P\in \sD_{\frX,\bbQ }^{\dagger}  | P \cJ  \subseteq \cJ \sD_{\frX,\bbQ }^{\dagger} \}.$$
\end{dfn} 
 
\begin{lemma}\label{lem-norm0} One has the following basic properties: 
\vskip8pt 
(i) $\sN(\cJ \sD_{\frX,\bbQ }^{\dagger})$ is a sub-$\cO_{\frX,\bbQ}$-ring of $\sD_{\frX,\bbQ }^{\dagger}$.
\vskip8pt 
(ii) $ \cJ \sD_{\frX,\bbQ }^{\dagger} \subseteq\sN(\cJ\sD_{\frX,\bbQ }^{\dagger})$ is a two-sided ideal.
\vskip8pt 
(iii) The quotient $\sN(\cJ\sD_{\frX,\bbQ }^{\dagger})/  \cJ \sD_{\frX,\bbQ }^{\dagger}$ is an $\cO_{\frX,\bbQ}/\cJ $-ring.
\end{lemma}
\begin{proof} Let $\sN:= \sN(\cJ \sD_{\frX,\bbQ }^{\dagger})$.
It is clear that  $\sN$ contains $\cO_{\frX,\bbQ}$ and is stable under addition. Given $P,Q\in \sN$, one has $$(PQ)\cJ\subseteq P(Q\cJ)\subseteq P ( \cJ \sD_{\frX,\bbQ }^{\dagger}) \subseteq (P\cJ) \sD_{\frX,\bbQ }^{\dagger}\subseteq \cJ \sD_{\frX,\bbQ }^{\dagger},$$
and so $PQ\in \sN$. This shows (i).
For (ii), note that  $\cJ \sD_{\frX,\bbQ }^{\dagger}$ is certainly contained in $\sN$ as a right ideal. 
If $P\in \sN$, then $P \cJ \sD_{\frX,\bbQ }^{\dagger} \subseteq \cJ \sD_{\frX,\bbQ }^{\dagger}$, so that  $\cJ\sD_{\frX,\bbQ }^{\dagger}$ is also a left ideal in $\sN$. This gives (ii). The point (iii) is a consequence of (i) and (ii) by observing that the induced ring homomorphism $$\cO_{\frX,\bbQ}\stackrel{\subseteq} {\rightarrow}  \sN\rightarrow \sN/  \cJ \sD_{\frX,\bbQ }^{\dagger}$$
factores through the quotient morphism $\cO_{\frX,\bbQ}\rightarrow \cO_{\frX,\bbQ}/\cJ$.
\end{proof}

\begin{cor}\label{cor-closed_sub1} Let 
$i : \frY \rightarrow \frX$ be a closed immersion between two smooth $\fro$-formal schemes, given by the ideal $\cI\subseteq\cO_\frX$.
There is a canonical isomorphism of $\cO_{\frY,\bbQ}$-rings 

$$\sD_{\frY,\bbQ }^{\dagger} \simeq i^{-1} \big( \sN(\cI_{\bbQ}\sD_{\frX,\bbQ }^{\dagger})/  \cI_{\bbQ} \sD_{\frX,\bbQ }^{\dagger}\big).
$$
\end{cor}
\begin{proof} According to \ref{prop_closed}, one has 
$$i_*\sD_{\frY,\bbQ }^{\dagger} \simeq \frac{ \{ P \in \sD_{\frX,\bbQ}^{\dagger} | P( \cI_\bbQ) \subseteq \cI_\bbQ \} } { \cI_\bbQ  \sD_{\frX,\bbQ}^{\dagger} }$$ as $\cO_{\frX,\bbQ}/\cJ $-modules. We claim that

$$\sN:= \sN(\cI_\bbQ\sD_{\frX,\bbQ }^{\dagger})= \{ P \in \sD_{\frX,\bbQ}^{\dagger} | P( \cI_\bbQ) \subseteq \cI_\bbQ \}.$$
If $P\in  \sN$ and $f\in\cI_\bbQ$, then there is $Q\in \sD_{\frX,\bbQ }^{\dagger}$ and $h\in \cI_\bbQ$ such that $Pf=hQ.$ It follows 
$P(f)=Pf(1)=hQ(1)\in h\cO_{\frX,\bbQ}\subseteq \cI_\bbQ$. This gives the forward inclusion. The inclusion being an equality may be checked locally. We may therefore assume that $\frX$ is affine. Let $A=\cO(\frX), I=\cO(\cI)$ and $\overline{A}=A/I$. One has $$N:= \cO(\sN)=
\{ P\in \sD_{A,\bbQ }^{\dagger}  | P I_\bbQ  \subseteq I_\bbQ \sD_{A,\bbQ }^{\dagger}\}.$$
The above isomorphism says that the restriction map $P\mapsto \overline{P}$, where $\overline{P}(\overline{a}):=P(a) + I_\bbQ$ for $\overline{a}\in\overline{A}_\bbQ$ induces a surjective $A_\bbQ$-linear morphism 

$$ res: \{ P \in \sD_{A,\bbQ}^{\dagger} | P( I_\bbQ) \subseteq I_\bbQ \} \longrightarrow \sD_{\overline{A},\bbQ}^{\dagger}$$
with kernel equal to $I_\bbQ  \sD_{A,\bbQ}^{\dagger}$. Let $P \in \sD_{A,\bbQ}^{\dagger}$ with $P( I_\bbQ) \subseteq I_\bbQ$ and take $f\in I_\bbQ$. Given $h\in I_\bbQ$, we have $Pf(h)=P(fh)\subseteq \cI_\bbQ$, since $fh\in I_\bbQ$, and hence $Pf( I_\bbQ) \subseteq I_\bbQ$. 
To calculate $res(Pf)$, we observe that $$Pf(A_\bbQ)=P(fA_\bbQ)\subseteq P (I_\bbQ)\subseteq I_\bbQ.$$Hence $res(Pf)=0$, whence $Pf\in I_\bbQ  \sD_{A,\bbQ}^{\dagger}$. This shows $P\in N$ and establishes the isomorphism of $\cO_{\frX,\bbQ}/\cI_\bbQ $-modules
$$i_*\sD_{\frY,\bbQ }^{\dagger} \simeq \frac{\sN(\cI_\bbQ\sD_{\frX,\bbQ }^{\dagger}) } { \cI_\bbQ  \sD_{\frX,\bbQ}^{\dagger} }.$$ However, since the map $res$ is a ring homomorphism, this isomorphism is in fact an isomorphism of $\cO_{\frX,\bbQ}/\cI_\bbQ $-rings. It remains to apply the functor $i^{-1}$.
\end{proof}

\subsection{$\beta$-bounded operators} 
In this subsection, $\frX$ is a smooth $\fro$-formal scheme endowed with local coordinates $x_1,...,x_M$. 
Let $A:=\cO(\frX)$. 
\vskip5pt 

The following basic result for local sections of $\sD^{\dagger}_{\frX}$ is \cite[2.4.4]{BerthelotDI}. Since its proof contains many arguments which we will refer to in the following, we will recall the proof here. This also allows us to fix some notations.
 
\begin{prop}\label{prop244} 
Let $||.||$ be a Banach norm on $A_\bbQ$. For any operator
$$ P=\sum_{\unu}a_{\unu}\uder^{[ \unu ]}= \sum_{\unu}a_{\unu}\uder^{\unu}/\unu !\in \Gamma(\frX,\hsD_{\frX})$$
we denote by $\overline{P}_i\in\Gamma(X_i, \sD_{X_i,k})$ the reduction of $P$ modulo $\varpi^{i+1}$. The following three conditions
are equivalent:
\vskip8pt
(i) $P\in \Gamma(\frX,\sD^{\dagger}_{\frX})$
\vskip8pt
(ii) There are constants $\alpha,\beta\in\bbR, \alpha>0$ such that for any $i\geq 0$
$$ \ord(\overline{P}_i) \leq \alpha i+\beta$$
\vskip8pt
(iii) There are constants $c,\eta\in\bbR, \eta<1$, such that for any $\unu$ $$ ||a_{\unu}||\leq c\eta^{|\unu|}.$$


\end{prop}

\begin{proof}
$(i)\Rightarrow (ii):$  Since $\frX$ is quasi-compact and quasi-separated, the functor $\Gamma(\frX,.)$ commutes with direct limits and we have $\Gamma(\frX,\sD^{\dagger}_{\frX})=\varinjlim_m \Gamma(\frX,\hsD_{\frX}^{(m)})$. We may therefore choose $m$ large enough and write $P$ as
$$ P=\sum_{\unu}b_{\unu}\uder^{\lan \unu \ran}=\sum_{\unu} \uq ! b_{\unu}\uder^{[ \unu ]}$$
where $b_{\unu}\in A$ tends to zero for $|\unu|\rightarrow\infty$, and where, for any $\nu\in\bbN$, we have written
$\nu=p^mq_{\nu}+r, 0\leq r<p^m$. Let $v$ be the normalized valuation of $\fro$. According to \cite[Lem. 2.4.3]{BerthelotDI}, there exists $\alpha',\beta'\in\bbR, \alpha'>0$ such that $v(q_{\nu}!)=ev_p(q_{\nu}!)\geq \alpha' \nu + \beta'$ for any $\nu\in\bbN$. Summing over all entries of $\unu$, we therefore find $\alpha,\beta''\in\bbR, \alpha>0$ such that $v(\uq !)\geq \alpha^{-1} |\unu| + \beta''$ for any $\unu\in\bbN^M$. Fix $i$. For any $\unu$ with
$|\unu|\geq \alpha (i-\beta''+1)$, the inequality $|\unu|\geq \alpha (i-\beta''+1)$ gives therefore
$$ v(\uq !)\geq i+1.$$ The latter means $\uq ! b_{\unu}\in \varpi^{i+1}A$ and so this coefficient of $P$ reduces to zero in $\overline{P}_i$. It follows
$$\ord(\overline{P}_i)\leq \alpha (i-\beta''+1)=\alpha i+\beta$$
where $\beta:= \alpha (1-\beta'')$.

\vskip8pt

$(ii)\Rightarrow (iii):$ To prove (iii), we may take for $||.||$
the spectral norm of $A_\bbQ$ which satisfies $||b||\leq 1$ for any $b\in A$.
Now suppose that $ \ord(\overline{P}_i) \leq \alpha i+\beta$ for all $i$.
Fix $\unu$. If $|\unu|>\alpha i+\beta$, one has $a_{\unu}\in\varpi^{i+1}A$, and hence $||a_{\unu}/\varpi^{i+1}||\leq 1$.
So for any $i\geq 0$ such that $|\unu|>\alpha i+\beta$, i.e. $\alpha^{-1}(|\unu|-\beta)>i$, one has $||a_{\unu}||\leq |\varpi|^{i+1}$. Now take $i$ such that $$i+1\geq \alpha^{-1}(|\unu|-\beta)>i.$$ Then, by what we have just shown,
$$ ||a_{\unu}||\leq |\varpi|^{\alpha^{-1}(|\unu|-\beta)}$$ for any $\unu$. With $\eta:=|\varpi|^{\alpha^{-1}}$ and $c:=|\varpi|^{-\beta\alpha^{-1}}$ we therefore have $||a_{\unu}||\leq c\eta^{|\unu|}$ for any $\unu$ where $\eta<1$.

\vskip8pt

$(iii)\Rightarrow (i):$ Suppose that $||a_{\unu}||\leq c \eta^{|\unu|}$ with $\eta< 1$.
Here, we may again assume that $||.||$ is the spectral norm on $A_\bbQ$. We first show that, for $m\geq 0$ sufficiently big, the elements
$b_{\unu}:=a_{\unu}/\uq !$ tend to zero for $|\unu|\rightarrow\infty$ in the Banach algebra $A_\bbQ$. By hypothesis we may write $\eta=p^{-\frac {\alpha}{e}}$ with $\alpha >0$ and obtain
$$ v_p(a_{\unu}):=-\log_p ||a_{\unu}||\geq |\unu| (-\log_p \eta)- \log_p c = |\unu|\left(\frac{\alpha}{e} \right)+\mu$$
o\`u $\mu=-\log_p c$.
An upper bound for the $p$-adic valuation of $\uq !$ is given by $|\unu|/p^m(p-1)$ \cite[Lem. 2.4.3]{BerthelotDI}. Hence,
$$ v_p(b_{\unu})=v_p(a_{\unu})-v_p(\uq !)\geq |\unu|\left(\frac{\alpha}{e}-\frac{1}{p^m(p-1)} \right)+\mu.$$
 Thus, for sufficiently big $m$, we obtain indeed $v_p(b_{\unu})\rightarrow \infty$ for $|\unu|\rightarrow\infty$.
 For the remaining statement, let $\lambda:=\alpha/e$. If $$|\unu|\geq -\mu(\lambda-1/p^m(p-1))^{-1},$$ then $v_p(b_{\unu})\geq 0$ and hence $||b_{\unu}||\leq 1$. This implies $b_{\unu}\in A$. Suppose therefore that $\mu<0$. Adjusting $m$ we may suppose that $p^m>\lambda^{-1}(-\mu+1/p-1).$ Then $$p^m>-\mu(\lambda-1/p^m(p-1))^{-1}.$$ So for all $$|\unu|< -\mu(\lambda-1/p^m(p-1))^{-1}$$ we obtain $\uq =0$, which implies $b_{\unu}=a_{\unu}\in A.$
\end{proof}

\vskip8pt

Let $n\geq 1$ and let $P\in M_n(  \Gamma(\frX,\hsD_{\frX}) )$ be a given $n\times n$-matrix, with entries in the ring
 $\Gamma(\frX,\hsD_{\frX})$. It will be convenient to write
$P=\sum_{\unu}a_{\unu}\uder^{[ \unu ]}$, with coefficients $a_{\unu}\in M_n(\Gamma(\frX,\cO_{\frX}))$. In particular, 
we may speak of the order $\ord(P)$ of $P$, whenever the sum is finite. In general, we write for any $\ell\geq 0$
$$ [P]_{\ell}:=\varpi^{-\ell}\sum_{\unu\in E_\ell} a_{\unu}\uder^{[ \unu ]}  \hskip5pt \textrm{ and }\hskip5pt
\sigma_{\ell}(P):=\sum_{\ell'\leq \ell} \varpi^{\ell'}[P]_{\ell'}$$
where $E_{\ell}:=\{\unu\in\bbN^M, v_p(a_{\unu})=\ell \}$ is a finite set.
Here,  $v_p(\cdot):=-\log_p ||\cdot||$ where $||\cdot||$ denotes any fixed Banach norm on $M_n(A_\bbQ)$. By definition one has
$$ \ord(\sigma_{\ell}(P))=\ord(\overline{P}_\ell)$$
where $\overline{P}_\ell$ is the reduction of $P$ modulo $\varpi^{\ell+1}$. This implies
$$
\ord(\sigma_{\ell}(P+Q))\leq\max ( \ord(\sigma_{\ell}(P)),\ord(\sigma_{\ell}(Q)))
$$
for two operators $P,Q$.
\begin{lemma}\label{equiv}
Let $\alpha>0$. One has the equivalence
$$\ord(\sigma_{\ell}(P))\leq \alpha\ell +\beta \hskip5pt \textrm{for all }\ell \hskip5pt \Leftrightarrow \hskip5pt \ord([P]_{\ell})\leq \alpha\ell +\beta  \hskip5pt \textrm{for all }\ell. \hskip5pt  $$
\end{lemma}
\begin{proof}
If $\ell'\neq\ell$, then $E_{\ell}$ is disjoint from $E_{\ell'}$. This implies the general identity
$$ \ord(\sigma_{\ell}(P))=\max_{\ell'\leq\ell}\;\ord(\varpi^{\ell'}[P]_{\ell'}).$$ In particular, the implication $\Rightarrow$ is clear.
On the other hand, if $[P]_{\ell'}\leq \alpha\ell'+\beta$ for any $\ell'$, then the right hand side in the above identity is bounded by $\alpha\ell+\beta$, since $\alpha>0$.
\end{proof}

\vskip8pt

\begin{dfn} Let $P\in M_n(  \Gamma(\frX,\hsD_{\frX}) )$ and $\beta >0$. The operator $P$ is called {\it $\beta$-bounded}, if for all $\ell\geq 0$

$$\ord(\sigma_{\ell}(P))\leq \beta ( \ell + 1).$$

\end{dfn}

\vskip8pt
\begin{lemma}\label{stable} Let $\beta>0$ and let $P_\ell$ be a $\varpi$-adically convergent sequence of $\beta$-bounded elements 
in $M_n(  \Gamma(\frX,\hsD_{\frX}) )$. Then
$P=\lim_\ell P_\ell$ is $\beta$-bounded.
\end{lemma}
\begin{proof}
Fix $n\geq 0$.
By the inequality before the lemma \ref{equiv}, we have $\ord (\sigma_n(P))\leq\max (\ord (\sigma_n(P-P_\ell)),\ord(\sigma_n(P_\ell)))$.
Choose $\ell$ sufficiently large such that $P-P_\ell$ is divisible by $\varpi^{n+1}$. Then $\sigma_n(P-P_\ell)=0$ and the claim follows.
\end{proof}

\vskip8pt

The interest in this notion comes from the following result.
\begin{prop}\label{bounded}
 Any $P\in M_n(\Gamma(\frX,\sD^{\dagger}_{\frX}))$ is $\beta$-bounded for some $\beta>0$. Conversely, any $\beta$-bounded $P\in M_n(\Gamma(\frX,\hsD_{\frX}))$ belongs to $M_n(\Gamma(\frX,\sD^{\dagger}_{\frX}))$.
\end{prop}

\begin{proof}
According to the above discussion,
$P$ is $\beta$-bounded iff $\ord([P]_\ell)\leq \beta ( \ell + 1)$ for all $\ell$ or, equivalently, if
$\ord(\overline{P}_\ell)\leq \beta ( \ell + 1)$ for all $\ell$. Thus, Prop. \ref{prop244} implies the claim.
\end{proof}

\subsection{A key lemma}\label{section_keylemma} 
In this subsection, $\frX$ is a smooth $\fro$-formal scheme endowed with local coordinates $x_1,...,x_M$. 
Let $A:=\cO(\frX)$. The main steps in the proof of the key lemma presented here are extracted from Caro's discussion in \cite{CaroSysInd}. 

\vskip8pt 
Let ${\bf 1}\in M_n( \Gamma(\frX,\hsD_{\frX}))$ be the identity in the matrix ring 
$M_n( \Gamma(\frX,\hsD_{\frX}))$. 
We consider $\Gamma(\frX,\hsD_{\frX})$ to be a subring of $M_n( \Gamma(\frX,\hsD_{\frX}))$ via
the injective ring homomorphism $$\Gamma(\frX,\hsD_{\frX})\longrightarrow M_n( \Gamma(\frX,\hsD_{\frX})), \; P\mapsto P{\bf 1}.$$

\vskip8pt

Let $m\geq 1$ and let $R\in M_n(\Gamma(\frX,\hsD_{\frX}^{(m)}))$ be a given $n\times n$-matrix. 

\vskip5pt 

By Prop. \ref{bounded}, $R$ is $\alpha$-bounded for some $\alpha>0$, so that
$\ord(\sigma_{\ell}(R))\leq \alpha ( \ell + 1)$
for all $\ell$.
 In the following we fix $\beta>0$ sufficiently large, such that
 $${\rm (HYP)}\hskip50pt  2\alpha +p^m\leq \beta.$$

We start with two auxiliary lemmas. Let $t:=x_M$. 

\begin{lemma}\label{lemcommutant} For any $U\in M_n(\Gamma(\frX,\sD_{\frX}))$ there is an operator $Q\in M_n(\Gamma(\frX,\sD_{\frX}))$ such that
$$
[t^{p^m},Q ]\equiv U {\rm~~ mod~~} \varpi \hskip10pt {\rm and}\hskip10pt \ord(Q)\leq \ord(U)+p^m.
$$
\end{lemma}
\begin{proof}
We first establish the general identity
$$ t^{p^m}\partial_t^{[N+p^m]}-\partial_t^{[N+p^m]} t^{p^m}=-\partial_t^{[N]} {\rm~~~ mod~~~} \varpi $$
for any integer $N$. Indeed, one has
$$ \partial_t^{[N+p^m]} t^{p^m}=\sum_{\nu+\nu'=N+p^m} \partial_t^{[\nu]}(t^{p^m}) \partial_t^{[\nu']}.$$
Since  $\partial_t^{[\nu]}(t^{p^m})=0$ for $\nu >p^m$ and $\equiv 0 \mod \varpi$ for
$0<\nu<p^m$, only the terms indexed by $(0,N+p^m)$ and $(p^m,N)$ survive in the sum. This yields the claimed identity. Now write $U=\sum_{\unu}a_{\unu}\uder^{[ \unu ]}$ and take
$Q=-\sum_{\unu}a_{\unu}\uder^{[ \unu +(0,...,0,p^m)]}$. Then $\ord(Q)\leq \ord(U)+p^m$ and
$$\begin{array}{ccc}

 [t^{p^m},Q ]&=&\sum_{\unu} (-a_{\unu})\uder^{[ (\nu_1,...,\nu_{M-1})]}
( t^{p^m}\partial_t^{[\nu_M+p^m]}-\partial_t^{[\nu_M+p^m]} t^{p^m})\\&\\
&\equiv&\sum_{\unu} (-a_{\unu})\uder^{[ (\nu_1,...,\nu_{M-1})]}(-\partial_t^{[\nu_M]}){\rm~~ mod~~} \varpi \\&\\
&\equiv & U.\\
\end{array}$$
\end{proof}

\

\begin{lemma}\label{lembounded} Let $Q_\ell\in M_n(\Gamma(\frX,\sD_{\frX}))$ such that $\ord(Q_\ell)\leq \beta (\ell+2)$. Then
$\varpi^{\ell+1}Q_\ell$ is $\beta$-bounded.
\end{lemma}
\begin{proof}
Let $P=\varpi^{\ell+1}Q_\ell$. If $i \leq\ell$, one has $\overline{P}_i=0$ and therefore
$\ord(\overline{P}_i)\leq \beta (i+1)$ trivially. If $i>\ell$, then
$$ \ord(\overline{P}_i)\leq \ord(P)\leq \beta (\ell+2)\leq \beta (i+1).$$
\end{proof}

\vskip8pt

We now construct by induction a sequence $P_\ell\in M_n(\Gamma(\frX,\sD_{\frX}))$ (depending on the matrix $R$) such that

\vskip8pt

(i) $P_0=1, \sigma_\ell(P_\ell)=P_\ell$ and
$P_{\ell+1}\equiv P_\ell$ mod $\varpi^{\ell+1}$

\vskip8pt

(ii) $P_\ell$ is $\beta$-bounded

\vskip8pt

(iii)
$t^{p^m}P_\ell \equiv P_\ell(t^{p^m}-\varpi\sigma_{\ell-1}(R))$ mod $\varpi^{\ell+1}$.

\vskip8pt

Suppose that $P_0,...,P_\ell$ are already constructed. We will construct $P_{\ell+1}$ in the following.

\vskip8pt

Since $\varpi\sigma_{\ell}\equiv \varpi\sigma_{\ell-1} {\rm~~~ mod~~~} \varpi^{\ell+1}$, the property (iii) implies
$$\begin{array}{ccccc}
[t^{p^m},P_\ell]+ \varpi P_\ell\sigma_\ell(R) &= & & & \\
& & & & \\
t^{p^m}P_\ell-P_\ell (t^{p^m} + \varpi \sigma_\ell(R))  &\equiv & t^{p^m}P_\ell-P_\ell (t^{p^m} + \varpi \sigma_{\ell-1}(R)) {\rm~~~ mod~~~} \varpi^{\ell+1}  &\equiv & 0 {\rm~~~ mod~~~} \varpi^{\ell+1}. \\
\end{array}
$$

\vskip8pt

On the other hand, since $\sigma_{\ell}(P_\ell)=P_{\ell}$ by (i), one has $P_{\ell}=\sum_{\ell_1\leq \ell}\varpi^{\ell_1}[P_{\ell}]_{\ell_1}$ and hence
$$ P_\ell \sigma_\ell(R)=\sum_{\ell_1,\ell_2\leq\ell} \varpi^{\ell_1+\ell_2}[P_\ell]_{\ell_1} [R]_{\ell_2}\equiv
\sum_{\ell_1,\ell_2\leq\ell, \ell_1+\ell_2\leq\ell+1} \varpi^{\ell_1+\ell_2}[P_\ell]_{\ell_1} [R]_{\ell_2}
{\rm~~~ mod~~~} \varpi^{\ell+2}.$$

So alltogether one obtains
$$ [ t^{p^m},P_\ell] +
\varpi \sum_{\ell_1,\ell_2\leq\ell, \ell_1+\ell_2\leq\ell+1} \varpi^{\ell_1+\ell_2}[P_{\ell}]_{\ell_1} [R]_{\ell_2}\equiv 0
 {\rm~~~ mod~~~} \varpi^{\ell+1}.$$

 So there is an operator $U_\ell\in M_n(\Gamma(\frX,\sD_{\frX}))$ such that

$$ -\varpi^{\ell+1} U_\ell = [ t^{p^m},P_\ell] +
 \varpi \sum_{\ell_1,\ell_2\leq\ell, \ell_1+\ell_2\leq\ell+1} \varpi^{\ell_1+\ell_2}[P_{\ell}]_{\ell_1} [R]_{\ell_2}.
$$

By the above discussion
$$-\varpi^{\ell+1} U_\ell \equiv [ t^{p^m},P_\ell] +  \varpi P_\ell\sigma_\ell(R)  {\rm~~~ mod~~~} \varpi^{\ell+2}.
$$

\vskip8pt

{\it Assertion 1}: One has $\ord(U_\ell)\leq \beta(\ell+1)+2\alpha.$

\vskip8pt

To prove the assertion, we use that $P_\ell$ is $\beta$-bounded by (ii).
In particular, $\ord(P_\ell)=\ord(\sigma_\ell(P_\ell))\leq \beta (\ell+1)$ which gives $\ord ([t^{p^m},P_\ell])\leq\beta (\ell+1)$.

Again, by (ii), we have $\ord([P_{\ell}]_{\ell_1}) \leq \beta (\ell_1+1)$ for all $\ell_1$. This gives
$$ \ord ([P_\ell]_{\ell_1} [R]_{\ell_2})\leq \ord ( [P_\ell]_{\ell_1}) + \ord ([R]_{\ell_2}) \leq \beta (\ell_1+1)+\alpha(\ell_2+1)\leq \beta(\ell+1)+2\alpha.$$

Note that the last inequality follows from $\alpha (\ell_2 -1) \leq \beta (\ell -\ell_1)$ which in turn follows directly from $\ell_1+\ell_2\leq\ell +1$ and $\alpha\leq\beta$. This implies the assertion.

\vskip8pt

We now use the lemma \ref{lemcommutant} to find an operator $Q_\ell\in M_n(\Gamma(\frX,\sD_{\frX}))$
such that

$$
[t^{p^m},Q_\ell ]\equiv U_\ell {\rm~~ mod~~} \varpi \hskip10pt {\rm and}\hskip10pt \ord(Q_\ell)\leq \ord(U_\ell)+p^m. \hskip30pt (*)
$$

\vskip5pt

We now set $$P_{\ell+1}:=\sigma_{\ell+1}(P_\ell +\varpi^{\ell+1}Q_\ell) \in M_n(\Gamma(\frU,\sD_{\frX})).$$

\vskip8pt

{\it Assertion 2}: The operator $P_{\ell+1}$ satisfies (i),(ii),(iii) above.

\vskip8pt

We start with (iii). Modulo $\varpi^{\ell+2}$ one certainly has the two congruences
$$ t^{p^m}P_{\ell+1}\equiv t^{p^m}(P_\ell +\varpi^{\ell+1}Q_\ell)$$
and
$$ P_{\ell+1}(t^{p^m}-\varpi\sigma_\ell(R))\equiv  (P_{\ell}+\varpi^{\ell+1}Q_\ell)(t^{p^m}-\varpi\sigma_\ell(R))$$ so it suffices to show that the two right-hand sides are congruent. But modulo $\varpi^{\ell+2}$, one has

$$\begin{array}{ccc}
 t^{p^m}(P_\ell +\varpi^{\ell+1}Q_\ell)&=&
t^{p^m}P_\ell +\varpi^{\ell+1}t^{p^m}Q_\ell\\&&\\

 &\equiv &t^{p^m}P_\ell +\varpi^{\ell+1}(U_\ell +Q_\ell t^{p^m})\\&&\\
& \equiv&
t^{p^m}P_\ell +\varpi^{\ell+1}Q_\ell t^{p^m} - ([ t^{p^m},P_\ell] +  \varpi P_\ell\sigma_\ell(R))\\&&\\
& =&(P_\ell +\varpi^{\ell+1}Q_\ell) t^{p^m} -  \varpi P_\ell\sigma_\ell(R)\\&&\\
& \equiv &(P_\ell +\varpi^{\ell+1}Q_\ell) (t^{p^m} -  \varpi \sigma_\ell(R))\\
\end{array}
$$
where the first congruence is the property (*) and the middle congruence is the congruence before assertion 1. 

\vskip8pt

To see (ii), we just note that our hypothesis {\rm (HYP)} implies  $\ord(Q_\ell)\leq \beta (\ell+2)$ by assertion 1. Hence, $\varpi^{\ell+1}Q_\ell$ is $\beta$-bounded by lemma \ref{lembounded}. Let $\ell'\leq\ell +1$. Since $\sigma_{\ell'}\circ\sigma_{\ell+1}=\sigma_{\ell'}$ we find
$$ \ord (\sigma_{\ell'}(P_{\ell+1}))=
\ord (\sigma_{\ell'}(P_{\ell}+\varpi^{\ell+1}Q_\ell))\leq\max (\ord (\sigma_{\ell'}(P_{\ell})),\ord (\sigma_{\ell'}(\varpi^{\ell+1}Q_\ell)))\leq \beta(\ell'+1)$$
where we have used the inequality before lemma \ref{equiv} and the fact that $P_\ell$ and  $\varpi^{\ell+1}Q_\ell$ are $\beta$-bounded. Let  $\ell'\geq\ell +1$. Since $\sigma_{\ell'}\circ\sigma_{\ell+1}=\sigma_{\ell+1}$ we find
$$ \ord (\sigma_{\ell'}(P_{\ell+1}))=  \ord (P_{\ell+1})
\leq\max (\ord (\sigma_{\ell+1}(P_{\ell})),\ord (\sigma_{\ell+1}(\varpi^{\ell+1}Q_\ell)))
\leq  \beta(\ell+1) \leq  \beta(\ell'+1).$$

Hence, $P_{\ell+1}$ is $\beta$-bounded.

\vskip8pt

It remains to see (i). The identity $P_{\ell+1}=\sigma_{\ell+1}(P_{\ell+1})$ is trivial.
In particular, we may write
$$ P_{\ell+1}=\sum_{v_p(a^{\ell+1}_{\unu})\leq\ell+1} a^{\ell+1}_{\unu}\uder^{[ \unu ]}$$
where $a^{\ell+1}_{\unu}\in M_n(\Gamma(\frU,\cO_{\frX}))$ are the coefficients of $P_{\ell+1}$. By definition, one has

$$a^{\ell+1}_{\unu}=a^{\ell}_{\unu}+\varpi^{\ell+1}\tilde{a}_{\unu}$$
where $\tilde{a}_{\unu}$ are the coefficients of $Q_\ell$. Because of $\sigma_{\ell}(P_\ell)=P_\ell$, one has $a^{\ell}_{\unu}=0$ for all coefficients
$a^{\ell}_{\unu}$ with
 $v_p(a^{\ell}_{\unu})\geq \ell+1$. In turn, the inequality $v_p(a^{\ell}_{\unu})\leq \ell$ implies
$v_p(a^{\ell+1}_{\unu})=v_p(a^{\ell}_{\unu})\leq\ell$ by the ultrametric inequality for $v_p$.

This means
$$ P_{\ell+1} {\rm~~~ mod~~~} \varpi^{\ell+1}\equiv \sum_{v_p(a^{\ell+1}_{\unu})\leq\ell} a^{\ell}_{\unu}\uder^{[ \unu ]}=
\sum_{v_p(a^{\ell}_{\unu})\leq\ell} a^{\ell}_{\unu}\uder^{[ \unu ]}=
\sigma_\ell(P_\ell)=P_\ell$$
which completes (i).

\vskip8pt

So there is indeed a sequence $(P_\ell)_\ell$ with the properties (i)-(iii) as claimed.
Choose $m'\geq m$ such that $P_\ell\in M_n(\Gamma(\frX,\sD_{\frX}^{(m')}))$ for all $\ell$
by prop. \ref{bounded}. We may consider its limit $$P=\lim_\ell P_\ell \in M_n(\Gamma(\frX,\hsD_{\frX}^{(m')})).$$ Then
we have inside $M_n(\Gamma(\frX,\hsD_{\frX}^{(m')}))$
\vskip8pt

(1) $P\equiv 1  {\rm~~~ mod~~~} \varpi$

\vskip5pt

(2) $t^{p^m}P = P(t^{p^m}-\varpi R).$

\vskip8pt

Thus we have proved the following lemma.
\begin{lemma}{\rm (Berthelot's key lemma)}\label{key_lemma}
Let $t:=x_M$ and let $R\in M_n(\Gamma(\frX,\hsD^{(m)}_{\frX}))$ be a given $n\times n$-matrix. There exists $m'\geq m$ and 
$P\in M_n(\Gamma(\frX,\hsD^{(m')}_{\frX}))$, such that

\vskip5pt

(1) $P\equiv 1  {\rm~~~ mod~~~} \varpi$

\vskip5pt

(2) $t^{p^m}P= P(t^{p^m}-\varpi R) \hskip8pt \text{~in~~}M_n(\Gamma(\frX,\hsD^{(m')}_{\frX}) ).$

\end{lemma}

\vskip5pt 

Note that, as a consequence of (1), the matrix $P$ appearing in the lemma is invertible in the 
$\varpi$-adically complete ring $M_n(\Gamma(\frX,\hsD^{(m')}_{\frX})$. 

\vskip5pt 

\begin{cor}\label{cor_key} Let $M$ be a finitely generated $\Gamma(\frX,\hsD^{(m)}_{\frX})$-module with generators $e_1,...,e_n$. Suppose that $t^{p^m}e_i \equiv 0 \mod \varpi$ for all $i$. Then there is $m'\geq m$ and a set of generators $e'_1,...,e'_n$ for the 
$\Gamma(\frX,\hsD^{(m')}_{\frX})$-module $\Gamma(\frX,\hsD^{(m')}_{\frX})\otimes_{ \Gamma(\frX,\hsD^{(m)}_{\frX})}M$ with the property
$t^{p^m}e_i' =0$ for all $i$.
\end{cor}
\begin{proof}
Write $\widehat{D}_{\frX}^{(m)}:=\Gamma(\frX,\hsD^{(m)}_{\frX})$ and $\underline{e}:=(e_1,...,e_n)$.
By assumption, there is
 $R\in M_n(\widehat{D}_{\frX}^{(m)})$ such that
 $$ t^{p^{m}}\underline{e}=\varpi R\underline{e}.$$ By the key lemma, there is $m'\geq m$ and
 $P\in M_n(\widehat{D}^{(m')}_{\frX})$, such that

\vskip8pt

(1) $P\equiv 1  {\rm~~~ mod~~~} \varpi$

\vskip8pt

(2) $t^{p^m}P = P(t^{p^m}-\varpi R) \hskip8pt \text{~in~~}M_n(\widehat{D}^{(m')}_{\frX}).$

\vskip8pt

Let $$M':=\widehat{D}_{\frX}^{(m')}\otimes_{\widehat{D}_{\frX}^{(m)}}  M \hskip10pt \text{and}\hskip10pt 
\underline{e'}:=P(1 \otimes \underline{e}).$$
Here, $1\otimes\underline{e}$ is the vector $( 1\otimes e_1,...,1\otimes e_n)\in (M')^n$, so that $$e_i'=\sum_{j} P_{ij} \otimes e_j\in M' \hskip10pt\text{ for } i=1,...,n.$$ The $e'_i$ are generators for the left
$\widehat{D}_{\frX}^{(m')}$-module $M'$. Indeed, given $y\in M'$ with $y=\sum \lambda_i\otimes e_i$, then $y=\sum \lambda'_i\otimes e'_i$ with 
$(\lambda'_1,...,\lambda'_n):= (\lambda_1,...,\lambda_n)\cdot P^{-1}$. Moreover, (2) implies
$$ t^{p^m} \underline{e'}=t^{p^m}P (1\otimes\underline{e})=P(t^{p^m}-\varpi R)(1\otimes\underline{e})
=P(1\otimes (t^{p^{m}}\underline{e}-\varpi R\underline{e}))=0.$$
\end{proof}

\section{The Berthelot-Kashiwara theorem}

Let $$i:\frY\longrightarrow \frX$$ be a closed immersion between smooth formal $\fro$-schemes given by an ideal sheaf $\cI \subseteq \cO_{\frX}$.

\subsection{Direct image and adjointness} Let $\sD^{\dagger}_{\frX\leftarrow\frY,\bbQ} $ be the associated transfer module, a $( i^{-1}  \sD^{\dagger}_{\frX,\bbQ},  \sD^{\dagger}_{\frY,\bbQ})$-bimodule, cf. \cite[3.4]{Berthelot_Dmod2}.
 Let $\sN$ be a left $\sD^{\dagger}_{\frY,\bbQ}$-module. Its direct image along $i$ is the left $\sD^{\dagger}_{\frX,\bbQ}$-module 
$$  i_{+}\sN:= i_* (\sD^{\dagger}_{\frX\leftarrow\frY,\bbQ}   \otimes_{\sD^{\dagger}_{\frY,\bbQ}}\sN ).$$
This yields a functor $i_{+}$ from left $\sD^{\dagger}_{\frY,\bbQ}$-modules to left $\sD^{\dagger}_{\frX,\bbQ}$-modules, cf. \cite[4.3.7]{BerthelotIntro}. If $ \frZ\stackrel{k}{\rightarrow} \frY$ is a second closed immersions of smooth 
formal $\fro$-schemes, then there is a natural isomorphism 
$(i\circ k)_+\simeq i_{+}\circ k_{+}$ of functors \cite[3.5.2]{BerthelotIntro}.

 \begin{prop} \label{lem-Dflat} 
(i) The right $ \sD^{\dagger}_{\frY,\bbQ}$-module $\sD^{\dagger}_{\frX\leftarrow\frY,\bbQ} $ is flat. 
\vskip8pt 
(ii) The functor $i_{+}$ is exact. 
\vskip8pt 
(iii) If $S\subseteq\frY$ is a closed subset and $\sN$ is a left $\sD^{\dagger}_{\frY,\bbQ}$-module supported on $S$, then $i_+\sN$ is supported on $S$.

\end{prop}
\begin{proof} (i) maybe proved be adapting the proof in the classical setting \cite[1.3.5]{Hotta} as follows.
 Fix a level $m\geq 0$. 
 By definition \cite[3.4.1]{Berthelot_Dmod2},
$$\sD^{(m)}_{\frY\leftarrow\frX}= i^{-1} \big( \sD^{(m)}_{\frX} \otimes_{\cO_\frX} \omega_{\frX}^{-1} \big)\otimes_{i^{-1}\cO_{\frX}}\omega_{\frY}
$$
where $\omega_{\frX}$ and $\omega_{\frY}$ are the modules of differentials of highest order on $\frX$ and $\frY$ respectively. 
Since (i) is a local question, we may from now on assume that $\frX$ is affine equipped with local coordinates $t_1,...,t_d\in \cO_{\frX}$ and that $\cI=(t_{r+1},...,t_d)$ for some $0\leq r<d$. Let $\der_i$ be the derivation relative to $t_i$. We identify 
$i^{-1}\omega_{\frX}^{-1} \otimes_{i^{-1}\cO_{\frX}}\omega_{\frY}$ with $\cO_{\frY}$ via the section 
 $$({\rm d}t_1\wedge\cdot\cdot\cdot\wedge {\rm d}t_d)^{\otimes-1}\otimes ({\rm d}t_1\wedge\cdot\cdot\cdot\wedge {\rm d}t_r).$$
Note that $D' :=\oplus_{\underline{\nu}}  \der_{1}^{\lan \nu_{1} \ran }\cdot\cdot\cdot \der_r^{\lan \nu_r \ran}\cO_{\frX} \subset \sD^{(m)}_{\frX}$ is a subring of 
$ \sD^{(m)}_{\frX}$. It is clear that $\sD^{(m)}_{\frX}\simeq \fro [\underline{\der}]^{(m)}\otimes_{\fro} D'$ as a right $D'$-module, where 
$\fro [\underline{\der}]^{(m)}$ equals the free
 $\fro$-module on the basis $\der_{r+1}^{\lan \nu_{r+1} \ran }\cdot\cdot\cdot \der_d^{\lan \nu_d \ran}$. It follows that 
$$\sD^{(m)}_{\frY\leftarrow\frX} \simeq  \fro [\underline{\der}]^{(m)}\otimes_{\fro} ( i^{-1}D' \otimes_{i^{-1}\cO_{\frX}}\cO_{\frY}).$$
It is easy to see that $i^{-1}D' \otimes_{i^{-1}\cO_{\frX}}\cO_{\frY}\simeq \sD^{(m)}_{\frY}$ as a right $\sD^{(m)}_{\frY}$-module. This means 
$$\sD^{(m)}_{\frX\leftarrow\frY} \simeq \fro[\underline{\der}]^{(m)}\otimes_{\fro} \sD^{(m)}_{\frY}$$
as right $\sD^{(m)}_{\frY}$-modules. In particular,  $\sD^{(m)}_{\frX\leftarrow\frY}$ is free, and hence flat, as a right $\sD^{(m)}_{\frY}$-module. According to \cite[2.2.2]{HPSS} the sheaf of rings $\sD^{(m)}_{\frY}$ is locally noetherian, hence \cite[3.2.4]{BerthelotDI} implies that
 $\widehat{\sD}^{(m)}_{\frX\leftarrow\frY}$
 is flat as a right $\widehat{\sD}^{(m)}_{\frY}$-module. Passing to the limit and inverting $p$, we finally see that $\sD^{\dagger}_{\frX\leftarrow\frY}$ is flat as a right $\sD^{\dagger}_{\frY,\bbQ}$-module, as claimed. The point (ii) follows from (i). 
 Finally (iii) follows from the fact that, for any abelian sheaf $\sK$ on $\frY$, the stalk of $i_*\sK$ at $x\in \frX$ equals $\sK_x$ if $x\in\frY$ and is zero otherwise. Hence, if $S\subset \frY$ is closed and $\sN$ is supported on $S$ and $x\in \frY\setminus S$, then $$(i_+\sN)_x= (\sD^{\dagger}_{\frX\leftarrow\frY,\bbQ}   \otimes_{\sD^{\dagger}_{\frY,\bbQ}}\sN )_x =\sD^{\dagger}_{\frX\leftarrow\frY,\bbQ,x}   \otimes_{\sD^{\dagger}_{\frY,\bbQ,x}}\sN_x =0.$$
 \end{proof}

 We define the following functor from left $\sD^{\dagger}_{\frX,\bbQ}$-modules to
 left $\sD^{\dagger}_{\frY,\bbQ}$-modules: 

$$i^\natural\sM:=
{\mathcal Hom}_{i^{-1}\sD^{\dagger}_{\frX,\bbQ}}(\sD^{\dagger}_{\frX\leftarrow\frY,\bbQ} ,i^{-1}\sM).$$

\begin{prop}\label{prop-adjointD} 

(i) The functor $i^\natural$ is right adjoint to $i_+$. In particular, $i^\natural$ is left-exact. 
\vskip5pt 

(ii) If $ \frZ\stackrel{k}{\rightarrow} \frY$ is a second closed immersions of smooth 
formal $\fro$-schemes, then there is a natural isomorphism
$(i\circ k)^\natural \simeq k^{\natural}\circ i^{\natural}$.

\vskip5pt 

(iii)  If $S\subseteq\frX$ is a closed subset and $\sM$ is a left $\sD^{\dagger}_{\frX,\bbQ}$-module supported on $S$, then $i^\natural \sM$ is supported on $S\cap\frY$.

  \end{prop}
\begin{proof} Since $i$ is a closed immersion, $i_*$ has the right adjoint $i^{-1}$. Therefore, for any left $\sD^{\dagger}_{\frY,\bbQ}$-module
$\sN$ and any left $\sD^{\dagger}_{\frX,\bbQ}$-module $\sM$, one has 
$$ {\mathcal Hom}_{\sD^{\dagger}_{\frX,\bbQ}} (i_{+}\sN, \sM)=
 {\mathcal Hom}_{i^{-1}\sD^{\dagger}_{\frX,\bbQ}} (\sD^{\dagger}_{\frX\leftarrow\frY,\bbQ}   \otimes_{\sD^{\dagger}_{\frY,\bbQ}} \sN, i^{-1}\sM). $$
 One obtains (i) by combining this with the standard tensor-hom adjunction 
 $$ {\mathcal Hom}_{i^{-1}\sD^{\dagger}_{\frX,\bbQ}} (\sD^{\dagger}_{\frX\leftarrow\frY,\bbQ}   \otimes_{\sD^{\dagger}_{\frY,\bbQ}} \sN, i^{-1}\sM)=  {\mathcal Hom}_{\sD^{\dagger}_{\frY,\bbQ}}(\sN,  {\mathcal Hom}_{i^{-1}\sD^{\dagger}_{\frX,\bbQ}}(\sD^{\dagger}_{\frX\leftarrow\frY,\bbQ}  ,i^{-1}\sM)).$$
 
 (ii) follows from uniqueness of adjoint functors and the fact that 
$(i\circ k)_+\simeq i_{+}\circ k_{+}$. For (iii), for any abelian sheaf $\sK$ on $\frX$, the stalk of $i^{-1}\sK$ at $x\in \frY$ equals $\sK_x$. Hence, if $S\subset \frX$ is a closed subset and $\sM$ is supported on $S$ and $x\in \frY\setminus S$, then $(i^\natural \sM)_x=0$. Indeed, 
this is a local statement, and we may suppose that the coherent module $\sM$ has a global finite presentation. 
This means that 
$i^{-1}\sM$ can be written as $\coker(f)$ for some $i^{-1}\sD^{\dagger}_{\frX,\bbQ}$-linear morphism 
$(i^{-1}\sD^{\dagger}_{\frX,\bbQ})^{\oplus s}\stackrel{f}{\longrightarrow} (i^{-1}\sD^{\dagger}_{\frX,\bbQ})^{\oplus t}$. 
Take an open $\frU\subseteq\frY$ containing $x$ with $\frU \cap S = \emptyset$. For any $y\in\frU$, one has $(i^{-1}\sM)_y=\sM_y=0.$
Hence $f_y$ is surjective for any $y\in\frU$ and so $f|_{\frU}$ is surjective, i.e. $i^{-1}\sM |_\frU = 0$. Taking the limit over all open neigbourhoods $\frV\subset \frU$ of $x$, one finds 

$$(i^\natural \sM)_x= \varinjlim_{x\in\frV \subset\frU} 
\Hom_{i^{-1}\sD^{\dagger}_{\frX,\bbQ}|_\frV }(\sD^{\dagger}_{\frX\leftarrow\frY,\bbQ}|_\frV ,i^{-1}\sM |_\frV ) =0.$$

\end{proof}

\subsection{Berthelot-Kashiwara for left modules}

We start with an auxiliary lemma. 

\begin{lemma}\label{lem_aux} Let $N$ be an $K$-vector space and $s \geq 1$. Let 
$N[[\uder^{[ \unu ]}]]$ be the $K$-subspace of all formal infinite sums 
$\sum_{\unu\in \bbN^s}n_{\unu}\uder^{[ \unu ]}$ with $n_{\unu}\in N$ on formal symbols $\uder^{[ \unu ]}$.
We regard $N$ as a $K$-subspace of $N[[\uder^{[ \unu ]}]]$ via $n\mapsto n \uder^{[ \underline{0} ]}$.
Define for $i=1,...,s$ a linear operator $t_i$ on $N[[\uder^{[ \unu ]}]]$ through 
$t_i \cdot (n_{\unu}\uder^{[ \unu ]}) :=  n_{\unu} (\uder^{[ \unu-e_i ]})$ when $\nu_i>0$ and zero else. Then  
$$\bigcap _{i=1,...,s} \ker \big(\;N[[\uder^{[ \unu ]}]]\stackrel{t_i\cdot }{\longrightarrow} N[[\uder^{[ \unu ]}]]\;\big)=N. $$
\end{lemma}
\begin{proof}
It suffices the check the forward inclusion, the reverse inclusion being clear.
By induction on $s$ it suffices to treat the case $s=1$. 
Writing $\sum_{\nu\in \bbN}m_{\nu}\der^{[ \nu ]}:=
t \cdot \sum_{\nu\in \bbN}n_{\nu}\der^{[ \nu ]}$, one has $m_\nu = n_{\nu+1}$. If $\sum_{\nu\in \bbN}n_{\nu}\der^{[ \nu ]}\in\ker(t)$, then
$n_{\nu+1}=m_\nu=0$ for all $\nu$. 
\end{proof}

Let ${\rm Coh}(\sD^{\dagger}_{\frY,\bbQ})$ and ${\rm Coh}^{\frY}(\sD^{\dagger}_{\frX,\bbQ})$ be the categories of coherent left
 $\sD^{\dagger}_{\frY,\bbQ}$-modules and coherent 
left  $\sD^{\dagger}_{\frX,\bbQ}$-modules with support in $\frY$, respectively. 
 
   \begin{prop}\label{prop_ff}
   (i) The functor $i_{+}$ restricts to a functor $i_+: {\rm Coh}(\sD^{\dagger}_{\frY,\bbQ})\rightarrow {\rm Coh}^{\frY}(\sD^{\dagger}_{\frX,\bbQ})$.
   \vskip5pt 
    (ii) 
  The unit $\eta_{\sN}: \sN \car (i^\natural \circ i_+)\sN$
  is an isomorphism for any $\sN\in  {\rm Coh}(\sD^{\dagger}_{\frY,\bbQ})$.
  \vskip5pt 
  (ii) The functor $i_{+}: {\rm Coh}(\sD^{\dagger}_{\frY,\bbQ})\rightarrow  {\rm Coh}^{\frY}(\sD^{\dagger}_{\frX,\bbQ})$
 is fully faithful. 
  \end{prop}
\begin{proof}
The preservation of coherence follows from \cite[4.3.8]{BerthelotIntro}.  Now (ii) is a local question and 
we may assume that $\frX$ is affine equipped with local coordinates $t_1,...,t_d\in \cO_{\frX}$ and that $\cI=(t_{r+1},...,t_d)$ for some $0\leq r<d$. Let $\der_i$ be the derivation relative to $t_i$. 
We identify 
$\omega_{\frX}$ with $\cO_{\frX}$ via the section ${\rm d}t_1\wedge\cdot\cdot\cdot\wedge {\rm d}t_d$ and similarly for $\omega_{\frY}$.
It then follows from the existence of the adjoint operator \cite[1.2.2]{Berthelot_Dmod2} and the fact that 
$$\sD^{\dagger}_{\frX\rightarrow\frY,\bbQ} =i^{*} \sD^{\dagger}_{\frX,\bbQ}=i^{-1}   ( \sD^{\dagger}_{\frX,\bbQ} / \cI  \sD^{\dagger}_{\frX,\bbQ} ),$$ that we have an isomorphism of $( i^{-1}  \sD^{\dagger}_{\frX,\bbQ},  \sD^{\dagger}_{\frY,\bbQ})$-bimodules
$$\sD^{\dagger}_{\frX\leftarrow\frY,\bbQ} \simeq i^{-1}   ( \sD^{\dagger}_{\frX,\bbQ} /  \sD^{\dagger}_{\frX,\bbQ} \cI ).$$
It follows that $\varphi\mapsto \varphi (1)$ gives a natural identification 
$$i^\natural\sM ={\mathcal Hom}_{i^{-1}\sD^{\dagger}_{\frX,\bbQ}}(\sD^{\dagger}_{\frX\leftarrow\frY,\bbQ} ,i^{-1}\sM)\simeq
 \bigcap _{i=r+1,...,d} \ker ( i^{-1}\sM\stackrel{t_i\cdot }{\longrightarrow} i^{-1}\sM ).$$

Suppose now that $\sM=i_{+}\sN$. According to \cite[3.4.5]{BerthelotDI} and \cite[3.6.2]{BerthelotDI}, there is an inductive system of 
coherent $\hsD^{(m)}_{\frY}$-submodules $\sN^{(m)}\subset \sN$, such that $ \varinjlim_{m}  \sN^{(m)}=\sN$. 
We have seen in the proof of  \ref{lem-Dflat}, that there is a natural isomorphism
$$\sD^{(m)}_{\frX\leftarrow\frY} \simeq \fro[\underline{\der}]^{(m)}\otimes_{\fro} \sD^{(m)}_{\frY}$$
as right $\sD^{(m)}_{\frY}$-modules for any $m$. 
If we define an (injective) transition map  $\fro[\underline{\der}]^{(m)}\rightarrow  \fro[\underline{\der}]^{(m+1)}$ as in \cite[2.2.3.1]{BerthelotDI}, then the isomorphism is compatible with the transition maps $\sD^{(m)}_{\frX\leftarrow\frY} \rightarrow \sD^{(m+1)}_{\frX\leftarrow\frY}$ 
and $\sD^{(m)}_{\frY}\rightarrow  \sD^{(m+1)}_{\frY}$. Hence, if $M=\cO(\sM), N=\cO(\sN)$ and $N^{(m)}=\cO(\sN^{(m)})$ and if 
$ \fro[\underline{\der}]^{(m)} \widehat{\otimes}_{\fro}\; N^{(m)}$ denotes the $\varpi$-adic completion of the $\fro$-module $\fro[\underline{\der}]^{(m)} \otimes_{\fro}\; N^{(m)}$, then 
$$M\simeq \varinjlim_{m} ( \fro[\underline{\der}]^{(m)} \widehat{\otimes}_{\fro}\; N^{(m)})_\bbQ .$$ 
Moreover, for $i=r+1,...,d$, the action of $t_i\in\cO_{\frX}$ on $M$ is given on the right-hand side by the action on the left-hand factor $\fro[\underline{\der}]^{(m)}$. Since  $\fro[\underline{\der}]^{(m)}$ is a free $\fro$-module, the tensor product  
$\fro[\underline{\der}]^{(m)} \otimes_{\fro}\; N^{(m)}$ is canonically isomorphic to the $\fro$-module of all finite formal sums
$\sum_{\unu\in \bbN^{d-r}}n_{\unu}\uder^{\langle \unu \rangle}$ with $n_{\unu}\in N^{(m)}$. Its completion $\fro[\underline{\der}]^{(m)} \widehat{\otimes}_{\fro}\; N^{(m)}$ is therefore isomorphic to the $\fro$-module given by all formal infinite sums $\sum_{\unu\in \bbN^{d-r}}n_{\unu}\uder^{\langle \unu \rangle}$ with $n_{\unu}\in N^{(m)}$ and $n_{\unu}\rightarrow 0$ in the $\varpi$-adic topology of $N^{(m)}$. Using the notation of \ref{lem_aux} with $s=d-r$, we obtain thus an injective $K$-linear map from $( \fro[\underline{\der}]^{(m)} \widehat{\otimes}_{\fro}\; N^{(m)})_\bbQ$ into $N[[\uder^{[ \unu ]}]]$. It is equivariant for the action of $t_i$ on $N[[\uder^{[ \unu ]}]]$ given by 
$t_i \cdot (n_{\unu}\uder^{[ \unu ]}) :=  n_{\unu} (\uder^{[ \unu-e_i ]})$ when $\nu_i>0$ and zero else. Passing to the limit over $m$ yields $K$-linear injection 
$$M\rightarrow N[[\uder^{[ \unu ]}]],$$
which is equivariant for the action of $t_i$ for all $i=r+1,...,d.$
According to the lemma \ref{lem_aux} we obtain 

$$N\subseteq  \bigcap _{i=r+1,...,d} \ker (M \stackrel{t_i\cdot }{\longrightarrow} M)\subseteq   
\bigcap _{i=r+1,...,d} \ker (N[[\uder^{[ \unu ]}]] \stackrel{t_i\cdot }{\longrightarrow} N[[\uder^{[ \unu ]}]])=N. $$
This implies $i^\natural\sM \simeq \sN$. Hence the unit of the adjunction is an isomorphism. This shows (i).
  The statement (i) implies immediately that $i_+$ is faitful. For the fullness, let
  $\gamma: i_+(\sN)\rightarrow i_+(\sN')$ be a morphism. 
  A preimage is given by the morphism 
$$\eta_{\sN'}^{-1} \circ i^\natural (\varepsilon_{i_+(\sN')} \circ i_+(\eta_{\sN'}) \circ
\gamma) \circ \eta_\sN: \; \sN\longrightarrow \sN'$$ where $\varepsilon : i_+ \circ i^\natural \to \text{id}$ is the counit of the
adjunction, cf. the proof of
  \cite[\href{https://stacks.math.columbia.edu/tag/07RB}{Tag 07RB}]{stacks-project}.
 \end{proof}

We now work towards the essential surjectivity of $i_{+}: {\rm Coh}(\sD^{\dagger}_{\frY,\bbQ})\rightarrow  {\rm Coh}^{\frY}(\sD^{\dagger}_{\frX,\bbQ})$.
 \begin{lemma}\label{lem_ess1}
 Let $\frY\subset \frX$ be of codimension $1$ and let $\sM\in {\rm Coh}^{\frY}(\sD^{\dagger}_{\frX,\bbQ})$.
 The counit of the adjunction $\varepsilon_\sM: (i_+\circ i^\natural) \sM \twoheadrightarrow \sM$ is surjective.

 \end{lemma}
 \begin{proof}
 The surjectivity of $\varepsilon_\sM$ is a local problem.
 We may therefore assume that $\frX$ is an open affine with local coordinates $x_1,...,x_M$ such that $\cI$ is generated by $t:=x_M$. 
 According to \cite[3.6]{BerthelotDI}, we can assume that there is $m\geq 0$ such that $$\sM\simeq \sD^{\dagger}_{\frX,\bbQ}\otimes_{\widehat{\sD}^{(m)}_{\frX,\Q}} \sM_m$$
 with a coherent $\widehat{\sD}^{(m)}_{\frX,\Q}$-module $\sM_m$ supported on $\frY$. 
 Let $\widehat{D}_{\frX}^{(m)}:=\Gamma(\frX,\hsD^{(m)}_{\frX})$ and $\widehat{D}_{\frX,\bbQ}^{(m)}:=\Gamma(\frX,\hsD^{(m)}_{\frX,\bbQ})$.
Let $M_m=\Gamma(\frX,\sM_m).$ Let $\mathring{M}_m\subset M_m$ be a finitely generated $\widehat{D}_{\frX}^{(m)}$-submodule such that 
 $\widehat{D}_{\frX,\bbQ}^{(m)} \mathring{M}_m=M_m$. Let $e_1,...,e_n\in \mathring{M}_m$ such that $\mathring{M}_m= \sum_i \widehat{D}_{\frX}^{(m)}e_i$.
The module $\mathring{M}_m/\varpi \mathring{M}_m$ has support contained in $\frY$ and hence $t^{p^{m}}\bar{e_i}=0$ for all $i$ (increasing $m$ if necessary), where $\bar{e_i}=e_i {\; \rm mod \;} \varpi \mathring{M}_m$. By \ref{cor_key}, there are generators $e'_1,...,e'_n$ for the 
$\widehat{D}_{\frX}^{(m')}$-module $\widehat{D}_{\frX}^{(m')} \otimes_{\widehat{D}_{\frX}^{(m)}}\mathring{M}_m
$ with the property $t^{p^m}e'_i=0$ for all $i$. 

Now $$\sM \simeq  \sD^{\dagger}_{\frX,\bbQ}\otimes_{\widehat{\sD}^{(m')}_{\frX,\Q}} \sM_{m'} \hskip10pt \text{ with } \hskip10pt \sM_{m'}:=\widehat{\sD}^{(m')}_{\frX,\Q} \otimes_{\widehat{\sD}^{(m)}_{\frX,\Q}} \sM_m.$$
Let $D^{\dagger}_{\frX,\bbQ} =\Gamma(\frX,   \sD^{\dagger}_{\frX,\bbQ}), M:=\Gamma(\frX,\sM)$ and $M_0:=\Gamma(\frX, i^{\natural}\sM)$. 
Since $\Gamma(\frX,\sM_{m'})= (\mathring{M}_{m'})_\Q$, it is clear that the $e_i'$ are generators for the $D^{\dagger}_{\frX,\bbQ}$-module
 $M$. As we have seen in the proof of \ref{prop_ff}, we have 
$M_0=\ker(t)\subseteq M.$ The counit $\varepsilon_\sM$ is therefore surjective, if and only if 
$D^{\dagger}_{\frX,\bbQ} M_0=M$. This is the case if $e'_i\in D^{\dagger}_{\frX,\bbQ} M_0$ for all $i$.

Since $t^{p^m}e'_i=0$ for all $i$, it suffices to show the following claim: given an element $u\in M$ with $t^ju=0$ for some $1\leq j\leq p^m$, then $u\in D^{\dagger}_{\frX,\bbQ}M_0.$ To prove the claim, we use a finite induction on $j$, the case $j=1$ being clear. So suppose $j>1$ and that the statement holds for $j-1$. We have
$$ t^{j-1}(ju+t\der u)=\der(t^j)u+t^j\der u=\der(t^ju)=\der(0)=0$$
and so, by induction hypothesis, $ju+t\der u\in D_{\frX,\bbQ}^{\dagger}M_0.$ Similarly,
$ t^{j-1}(tu)=0$ implies $tu\in D_{\frX,\bbQ}^{\dagger}M_0$ and hence also $-\der (tu)\in D_{\frX,\bbQ}^{\dagger}M_0$.
Alltogether,
$$ (j-1)u=ju-u=ju+(t\der - \der t)u =ju+t\der u -\der (tu)\in D_{\frX,\bbQ}^{\dagger}M_0.$$
and it remains to divide by $j-1$. This completes the induction step and establishes the equality $D_{\frX,\bbQ}^{\dagger} M_0=M$. Hence 
the proposition is proved.
\end{proof}

\begin{thm}\label{thm_BK} {\rm (Berthelot-Kashiwara theorem, left version)}
The functors $i_+,i^\natural$ induce mutually inverse equivalences of categories
$$
\xymatrix{
 {\rm Coh}(\sD^{\dagger}_{\frY,\bbQ}) \;\; \ar@<1ex>[r]^{i_+} &\;\; {\rm Coh}^{\frY}(\sD^{\dagger}_{\frX,\bbQ})  \ar@<1ex>[l]^{i^\natural}_{\simeq}
}.
$$
\end{thm}
\begin{proof} We first suppose that $\frY\subset \frX$ be of codimension $1$.
It suffices to show that the counit of the adjunction $\varepsilon_\sM: (i_+\circ i^\natural) \sM \twoheadrightarrow \sM$ is an isomorphism for any $\sM$. This is a local question and we may assume that $\frX$ is affine with coordinates $x_1,...,x_d$ and 
that $\cI$ is generated by $t:=x_d$. Since $(i_+\circ i^\natural)\sM \twoheadrightarrow \sM$ is surjective according to \ref{lem_ess1} we may also assume that $\sM$ is globally generated, as $\sD^{\dagger}_{\frX,\bbQ}$-module, by finitely many sections $e_1,...,e_n\in  i^\natural\sM=\ker(t) \subseteq \sM$. Hence there is a free $\sD^{\dagger}_{\frY,\Q}$-module $\sL$ of rank $n$ and a linear surjection  $$i_+\sL\twoheadrightarrow\sM.$$ Let $\sK$ be the kernel of this morphism, a coherent
$\sD^{\dagger}_{\frX,\Q}$-module with support in $\frY$. The morphism $ i_+ i^\natural \sK\twoheadrightarrow \sK$
is surjective, again by \ref{lem_ess1}. Similarly to the above,
$\sK$ is therefore globally generated, as $\sD^{\dagger}_{\frX,\bbQ}$-module, by finitely many sections $f_1,...,f_m \in  i^\natural\sK=\ker(t) \subseteq \sK$. Consider the
$\sD^{\dagger}_{\frY,\Q}$-submodule $$\sV:= \sum_j \sD^{\dagger}_{\frY,\Q} f_j \subseteq  i^\natural\sK.$$ By construction, the composite map $i_+\sV \rightarrow i_+i^\natural \sK\rightarrow \sK$ is a 
linear surjection $$i_+\sV\twoheadrightarrow \sK.$$ 
Moreover,
$$ \sV \subseteq i^\natural\sK\subseteq i^\natural i_+\sL \simeq \sL,$$
where the second inclusion holds by left-exactness of $i^\natural$, cf. \ref{prop-adjointD}, and the final isomorphism holds by \ref{prop_ff}.
Hence, $i_+\sV\hookrightarrow i_+\sL$ is injective with image $\sK$. All in all,
$$ i_+(\sL/\sV)\simeq i_+\sL/i_+\sV\simeq i_+\sL/\sK\simeq \sM.$$ 
The $\sD^{\dagger}_{\frY,\Q}$-module $\sL/\sV$ is finitely presented and hence coherent. So $i_+$ is essentially surjective.
Moreover, $\sL/\sV\simeq (i^\natural \circ i_+)(\sL/\sV)\simeq i^\natural \sM$. So the functor $i^\natural$ takes ${\rm Coh}^{\frY}(\sD^{\dagger}_{\frX,\bbQ})$ into $ {\rm Coh}(\sD^{\dagger}_{\frY,\bbQ})$ and is a quasi-inverse to $i_+$. This proves the theorem in case of codimension $1$. 

In the general case, we again reduce to the case where $\frX$ is affine with coordinates $x_1,...,x_d$ and 
that $\cI$ is generated by $x_{r+1},...,x_d$. Define a series of closed subschemes of $\frX$ by 
$\frY_{k}:=V(x_{r+1},...,x_{d-{k-1}})$ for $k=1,...,d-r$, i.e.
$$\frY=\frY_1 \subset \frY_{2} \subset \cdot\cdot\cdot \frY_{d-r-1} \subset \frY_{d-r} \subset \frX$$
and $i_k: \frY_k \subset \frY_{k+1}$ is a closed immersion between smooth formal schemes of codimension $1$. 
We use a finite induction on $k$. We call $(S_k)$ the following statement :\vskip5pt 
$(i_k)_+ \circ \cdot \cdot (i_1)_+$ induces an equivalence of categories between 
${\rm Coh}(\sD^{\dagger}_{\frY,\bbQ})$ and  ${\rm Coh}^{\frY}(\sD^{\dagger}_{\frY_k,\bbQ})$ with quasi-inverse 
$(i_1)^\natural \circ \cdot \cdot (i_k)^\natural$. 

\vskip5pt

By the codimension $1$ case, the functor $(i_{k+1})_+$ is an equivalence of categories between 
${\rm Coh}(\sD^{\dagger}_{\frY_k,\bbQ})$ and  ${\rm Coh}^{\frY_k}(\sD^{\dagger}_{\frY_{k+1},\bbQ})$ 
with quasi-inverse $(i_{k+1})^\natural.$ In particular, $(S_1)$ is true. Suppose that $(S_k)$ is true.
 According to \ref{lem-Dflat}(iii) and \ref{prop-adjointD}(iii), the functor $(i_{k+1})_+$ restricts to an equivalence between objects supported on $\frY$, i.e. to an equivalence
 ${\rm Coh}^{\frY}(\sD^{\dagger}_{\frY_k,\bbQ})$ and  ${\rm Coh}^{\frY}(\sD^{\dagger}_{\frY_{k+1},\bbQ})$
with quasi-inverse $(i_{k+1})^\natural.$ This establishes $(S_{k+1})$. The statement $(S_{d-r})$ then establishes the theorem, since 
$i_+=(i_{d-r})_+ \circ \cdot \cdot (i_1)_+$ and 
$i^\natural = (i_1)^\natural \circ \cdot \cdot (i_{d-r})^\natural$, the latter by \ref{prop-adjointD}.
\end{proof}

\subsection{Side-changing}\label{subsection_right}
We deduce the right version by using the side-changing functors 
$\omega_{\frX}\otimes_{\cO_{\frX}}(-)$ and 
  $\cH om_{\cO_{\frX}}(\omega_{\frX},-)$, cf. \ref{side_changing}. We consider the $( \sD^{\dagger}_{\frY,\bbQ}, i^{-1}\sD^{\dagger}_{\frX,\bbQ})$-bimodule
$\sD^{\dagger}_{\frX\rightarrow\frY,\bbQ}:=i^{*}\sD^{\dagger}_{\frX,\bbQ}$.
  
  Denote by $i_{r,+}$ the functor from right $\sD^{\dagger}_{\frY,\bbQ}$-modules to
 right $\sD^{\dagger}_{\frX,\bbQ}$-modules given by 
 $$ i_{r,+}(\sN) := \sN \otimes_{ \sD^{\dagger}_{\frY,\bbQ}} \sD^{\dagger}_{\frX\rightarrow\frY,\bbQ}$$
 
 as well as the functor from right $\sD^{\dagger}_{\frX,\bbQ}$-modules to
 right $\sD^{\dagger}_{\frY,\bbQ}$-modules given by 
  $$ i_{r}^{\natural}(\sM) :={\mathcal Hom}_{i^{-1}\sD^{\dagger}_{\frX,\bbQ}}(\sD^{\dagger}_{\frX\rightarrow\frY,\bbQ} ,i^{-1}\sM).$$

 \begin{lemma}\label{right_version} (i) The left $\sD^{\dagger}_{\frY,\bbQ}$-module 
 $\sD^{\dagger}_{\frX\rightarrow\frY,\bbQ}$ is flat. 
 \vskip5pt (ii) One has a natural isomorphism $i_{r,+}(\sN)\simeq  \omega_\frX \otimes_{\cO_{\frX}} i_+( \cH om_{\cO_{\frY}}(\omega_{\frY}, \sN)).$
 \vskip5pt (iii) One has a natural isomorphism  $i_{r}^{\natural}(\sM) \simeq  \omega_\frY \otimes_{\cO_{\frY}} i^{\natural}( \cH om_{\cO_{\frX}}(\omega_{\frX}, \sM)).$
 \end{lemma}
 \begin{proof} Given the definition of the transfer module \cite[3.4.1]{Berthelot_Dmod2}
 $$ \sD^{\dagger}_{\frX\leftarrow\frY,\bbQ}= i^{-1} \big( \sD^{\dagger}_{\frX,\bbQ} \otimes_{\cO_\frX} \omega_{\frX}^{-1} \big)\otimes_{i^{-1}\cO_{\frX}}\omega_{\frY},
$$
part (ii) follows formally exactly as in the classical case \cite[Lemma 1.3.4]{Hotta}. 
 Since $i_+$ is exact by \ref{lem-Dflat}, so is $i_{r,+}$, and then (ii) implies the flatness of the left $\sD^{\dagger}_{\frY,\bbQ}$-module 
 $\sD^{\dagger}_{\frX\rightarrow\frY,\bbQ}$ (one may also give a direct argument along the lines of the proof of part (i) of \ref{lem-Dflat}). This shows (i). 
 Finally, as in the proof of part (i) of \ref{prop-adjointD}, one verifies that ${\mathcal Hom}_{i^{-1}\sD^{\dagger}_{\frX,\bbQ}}(\sD^{\dagger}_{\frX\rightarrow\frY,\bbQ} ,i^{-1}(-))$ is right adjoint to  $(-) \otimes_{ \sD^{\dagger}_{\frY,\bbQ}} \sD^{\dagger}_{\frX\rightarrow\frY,\bbQ}$ and so (iii) follows from unicity of adjoint functors.
  \end{proof}
Let ${\rm Coh}^r(\sD^{\dagger}_{\frY,\bbQ})$ and ${\rm Coh}^{r,\frY}(\sD^{\dagger}_{\frX,\bbQ})$ be the categories of coherent 
 right $\sD^{\dagger}_{\frY,\bbQ}$-modules and coherent 
 right $\sD^{\dagger}_{\frX,\bbQ}$-modules with support in $\frY$, respectively. 

\begin{thm}\label{thm_BK_r} {\rm (Berthelot-Kashiwara theorem, right version)} The functors $i_{r,+},i_r^\natural$ induce mutually inverse equivalences of categories
$$
\xymatrix{
 {\rm Coh}^r(\sD^{\dagger}_{\frY,\bbQ}) \;\;\ar@<1ex>[r]^{i_{r,+}} &\;\;  {\rm Coh}^{r,\frY}(\sD^{\dagger}_{\frX,\bbQ})  \ar@<1ex>[l]^{i_r^\natural}_{\simeq}
}.
$$
\end{thm}
\begin{proof} Taking into account \ref{supp_condition}, this is a consequence of \ref{thm_BK} and \ref{right_version}.
\end{proof}

\section{The Berthelot-Kashiwara theorem in the twisted case}

\subsection{Twisted sheaves} 

The following definition, adapted to arithmetic differential operators, is taken from \cite{BB81}. 

\begin{dfn} {\it A twisted sheaf of arithmetic differential operators} on $\frX$ is an $\cO_{\frX,\bbQ}$-ring $\sA$, which is locally isomorphic to the $\cO_{\frX,\bbQ}$-ring $\sD^{\dagger}_{\frX,\bbQ}$.
\end{dfn}

Let in the following $\sA$ be a twisted sheaf of arithmetic differential operators on $\frX$.

\begin{lemma} \label{opposite_ring}
Let $\sA^{opp}$ be the opposite ring, i.e. the order of multiplication is reversed. Then $\sA^{opp}$ is an $\cO_{\frX,\bbQ}$-ring.
\end{lemma}
\begin{proof} Being a local statement, we can assume that $\frX$ has \'etale coordinates. 
The existence of the adjoint operator says that 
$(\sD^{\dagger}_{\frX,\bbQ})^{opp}$ is isomorphic to $\sD^{\dagger}_{\frX,\bbQ}$ \cite[1.2.2/3]{Berthelot_Dmod2}. 
Since this holds even as $\cO_{\frX,\bbQ}$-rings, the lemma follows. 
\end{proof}

\begin{dfn}\label{def_norm} Let $\cJ\subseteq\cO_{\frX,\bbQ}$ be an ideal. The {\it normalizer} of the ideal $\cJ \sA$ is the subset of $\sA$ equal to 
$$\sN_{\sA }(\cJ \sA):=\{ P\in \sA | P \cJ  \subseteq \cJ \sA  \}.$$
\end{dfn} 
 
\begin{lemma}\label{lem-norm2} One has the following basic properties: 
\vskip5pt 
(i) $\sN_\sA(\cJ \sA)$ is a sub-$\cO_{\frX,\bbQ}$-ring of $\sA$.
\vskip5pt 
(ii) $ \cJ \sA \subseteq\sN_\sA (\cJ\sA)$ is a two-sided ideal.
\vskip5pt 
(iii) The quotient $\sN_\sA(\cJ\sA) /  \cJ \sA $ is an $\cO_{\frX,\bbQ}/\cJ $-ring.
\vskip5pt
(iv) $\sA/\cJ \sA$ is a $(\sN_\sA(\cJ\sA) /  \cJ \sA, \sA)$-bimodule.

\end{lemma}
\begin{proof} The proof of (i)-(iii) is identical to the proof of \ref{lem-norm0}. The point (iv) is easy to check. 
\end{proof}

Now let $i: \frY\longrightarrow\frX$ be a smooth closed formal subscheme defined by some coherent ideal $\cI\subseteq \cO_{\frX}$. 
According to \ref{lem-norm2} the sheaf on $\frY$
$$\sA_{\frY}:=i^{-1} \big(\sN_\sA(\cI_\bbQ\sA)/\cI_\bbQ \sA\big)$$
 is an $\cO_{\frY,\bbQ}$-ring.
 
 \begin{lemma}\label{lem-AD} If $\sA= \sD^{\dagger}_{\frX,\bbQ}$, then $\sA_{\frY}=\sD^{\dagger}_{\frY,\bbQ}$.
\end{lemma}
\begin{proof} This is \ref{cor-closed_sub1}.
\end{proof}
\begin{cor} \label{cor-twist}
$\sA_{\frY}$ is a twisted sheaf of arithmetic differential operators on $\frY$. 
\end{cor}
Before we proceed, we recall an elementary lemma.
\begin{lemma} \label{lem-aux} Let $X$ be a topological space and $\cF_1,\cF_2\subseteq \cG$ abelian sheaves. If $\cF_{1,x}=\cF_{2,x}$ inside the stalk $\cG_x$ for all $x\in X$, then $\cF_1=\cF_2$.
\end{lemma}
\begin{proof} By symmetry, it suffices to see $\cF_1\subseteq\cF_2$.
The stalk of the sheaf $(\cF_1+\cF_2)/\cF_2$ vanishes at every $x\in X$, so that
 $(\cF_1+\cF_2)/\cF_2=0$, i.e. $\cF_1\subseteq\cF_2$.
\end{proof}
\begin{prop}\label{prop-comp1} Let $ \frZ\stackrel{k}{\rightarrow} \frY\stackrel{i}{\rightarrow}\frX$ be closed immersions of smooth 
formal $\fro$-schemes. There is a canonical isomorphism 
$(\sA_{\frY})_{\frZ}\simeq \sA_{\frZ}$ as $\cO_{\frZ,\bbQ}$-rings.

\end{prop}
\begin{proof}Let $\cJ\subseteq\cO_{\frX}$ be the ideal defining the closed immersion $i\circ k: \frZ\rightarrow \frX$. In particular, $\cI\subseteq \cJ$. Let $\overline{\cJ}$ be the image of $i^{-1}(\cJ/\cI)$ in $\cO_{\frY,\bbQ}$. By construction, the rings $(\sA_{\frY})_{\frZ}$ and $\sA_{\frY}$ are contained in $k^*(\sA_{\frY} )$ and $i^*(\sA)$ respectively. 
Hence there is an injective ring homomorphism 
$$\varphi: (\sA_{\frY})_{\frZ}\longrightarrow (i\circ k)^* (\sA)$$
compatible with the $\cO_{\frZ}$-structures. It suffices to see $\im (\varphi )=\sA_{\frZ}$. This is a local question, by \ref{lem-aux},
and so we may assume $\sA\simeq \sD^{\dagger}_{\frX,\bbQ}$. Then  $\sA_{\frY}\simeq  \sD^{\dagger}_{\frY,\bbQ}$ and 
$\sA_{\frZ}\simeq  \sD^{\dagger}_{\frZ,\bbQ}$ by \ref{lem-AD} and $\varphi$ becomes the canonical injective morphism 
$$k^{-1}\big(  \sN_{ \sD^{\dagger}_{\frY,\bbQ}} (\overline{\cJ}\sD^{\dagger}_{\frY,\bbQ}) /  \overline{\cJ}\sD^{\dagger}_{\frY,\bbQ}\big) \rightarrow (i\circ k)^* (\sD^{\dagger}_{\frX,\bbQ}).$$ Its image equals $\sD^{\dagger}_{\frZ,\bbQ}$, according to \ref{cor-closed_sub1}.
 \end{proof}

In the following we take the canonical isomorphism $(\sA_{\frY})_{\frZ} \simeq \sA_{\frZ}$ as an identification.

\subsection{Direct image and the main theorem} Let  
$$i:\frY\longrightarrow \frX$$ be a closed immersion between smooth formal $\fro$-schemes
 defined by some ideal $\cI\subset \cO_{\frX}$. Let $\sA$ be a twisted sheaf of arithmetic differential operators on $\frX$.
 We have the $(\sN_\sA(\cI \sA) /  \cI \sA, \sA)$-bimodule $\sA/ \cI \sA$, cf. \ref{lem-norm2}. 

\begin{dfn}
The {\it transfer bimodule along i} is the $(\sA_{\frY},i^{-1}\sA)$-bimodule
$$\sA_{\frY\rightarrow\frX}:=i^{*} (\sA)=i^{-1}(\sA/ \cI \sA). $$ 
\end{dfn}
\begin{prop}\label{prop-comp3}  Let $ \frZ\stackrel{k}{\rightarrow} \frY\stackrel{i}{\rightarrow}\frX$ be closed immersions of smooth 
formal $\fro$-schemes. There is a natural isomorphism 
as $(\sA_{\frZ}, (i\circ k)^{-1}\sA)$-bimodules 
$$\sA_{\frZ\rightarrow\frY}
\otimes_{k^{-1}\sA_{\frY}} k^{-1}\sA_{\frY\rightarrow\frX} \simeq \sA_{\frZ\rightarrow\frX}.$$
\end{prop}
\begin{proof} One has $$\begin{array}{ccl} 
\sA_{\frZ\rightarrow\frY}
\otimes_{k^{-1}\sA_{\frY}} k^{-1}\sA_{\frY\rightarrow\frX}& \simeq &  ( \cO_{\frZ}\otimes_{k^{-1}\cO_{\frY}} k^{-1} \sA_{\frY}  )
\otimes_{k^{-1}\sA_{\frY}} k^{-1} ( \cO_{\frY} \otimes_{i^{-1}\cO_{\frX}}  i^{-1}\sA)  \\&&\\
&\simeq &( \cO_{\frZ}  \otimes_{k^{-1}\cO_{\frY}}k^{-1} \sA_{\frY} ) \otimes_{k^{-1}\sA_{\frY}}  ( k^{-1} \cO_{\frY}   \otimes_{(i\circ k)^{-1}\cO_{\frX}}   (i\circ k)^{-1}\sA )   \\ &&\\
&\simeq &   \cO_{\frZ} \otimes_{(i\circ k)^{-1}\cO_{\frX}} (i\circ k)^{-1}\sA =\sA_{\frZ\rightarrow\frX}.
\end{array}$$
In the middle isomorphism we have use the compatibility of inverse images with tensor products, e.g. \cite[C.1.11(i)]{Hotta}.
\end{proof}

\begin{dfn} 
Let $\sM$ be a right $\sA_{\frY}$-module. Its {\it direct image along $i$} is the right $\sA$-module 
$$  i_{\sA,+}\sM:= i_* (\sM \otimes_{{\sA}_{\frY}}\sA_{\frY\rightarrow\frX} ).$$
\end{dfn}
It is clear that this yields a functor $i_{\sA,+}$ from right $\sA_{\frY}$-modules to right $\sA$-modules. 

\begin{lemma} \label{lem-Aflat} 
\vskip8pt 
(i)
The left $\sA_{\frY}$-module $\sA_{\frY\rightarrow\frX}$ is flat. 
\vskip8pt 
(ii) The functor $i_{\sA,+}$ is exact. 
\end{lemma}
\begin{proof}
Part (i) is a local question and so we may assume $\sA=\sD^{\dagger}_{\frX}$ and $\sA_{\frY\rightarrow\frX}=
\sD^{\dagger}_{\frY\rightarrow \frX}$. The claim follows then from part (i) of \ref{right_version}. Since $i_*$ is exact,
(i) implies (ii). 
\end{proof}

\begin{prop}\label{prop-comp2} Let $ \frZ\stackrel{k}{\rightarrow} \frY\stackrel{i}{\rightarrow}\frX$ be closed immersions of smooth 
formal $\fro$-schemes. There is a natural isomorphism 
$(i\circ k)_{\sA,+}=i_{\sA,+}\circ k_{\sA,+}$ as functors from right
$\sA_{\frZ}$-modules to right $\sA$-modules. 
\end{prop}
\begin{proof} Using the \ref{prop-comp3} one finds $$\begin{array}{ccl}

(i\circ k)_{\sA,+}(\sM) & \simeq & (i \circ k)_*(\sM \otimes_{\sA_{\frZ}}  \sA_{\frZ\rightarrow\frX}  ) \\ &&\\
& \simeq& i_* ( k_* (\sM \otimes_{\sA_{\frZ}}  ( \sA_{\frZ\rightarrow\frY}
\otimes_{k^{-1}\sA_{\frY}} k^{-1}\sA_{\frY\rightarrow\frX}  ) \\&&\\

& \simeq & i_* (k_*(\sM \otimes_{\sA_{\frZ}}\sA_{\frZ\rightarrow\frY}  )\otimes_{\sA_{\frY}}  \sA_{\frY\rightarrow\frX} )\\&&\\
& \simeq & i_{\sA,+} ( k_{\sA,+}(\sM) ).
\end{array}$$
In the third isomorphism, we have used the projection formula \cite[C.1.11(iii)]{Hotta}. For this, note that $k_{!}=k_*$ and that $\sA_{\frY\rightarrow\frX}$ is flat as left $\sA_{\frY}$-module \ref{lem-Aflat}. \end{proof}

We define the following functor from right $\sA$-modules to
 right $\sA_{\frY}$-modules: 

$$i_\sA^\natural\sM:=
{\mathcal Hom}_{i^{-1}\sA}(\sA_{\frY\rightarrow\frX},i^{-1}\sM).$$

\begin{prop}\label{prop-adjoint} The functor $i_{\sA,+}$ has a right adjoint, given by the functor $i_\sA^\natural$.
  \end{prop}
\begin{proof}
This is a sort of right version of the argument given in part (i) of \ref{prop-adjointD}. Since $i$ is a closed immersion, $i_*$ has the right adjoint $i^{-1}$. Therefore, for any right $\sA_{\frY}$-module
$\sN$ and any right $\sA$-module $\sM$, one has 
$$ {\mathcal Hom}_{\sA} (i_{\sA,+}\sN, \sM)=
 {\mathcal Hom}_{i^{-1}\sA} (\sN \otimes_{{\sA}_{\frY}}\sA_{\frY\rightarrow\frX}, i^{-1}\sM). $$
 One obtains (i) by combining this with the standard tensor-hom adjunction 
 $$ {\mathcal Hom}_{i^{-1}\sA} (\sN \otimes_{{\sA}_{\frY}}\sA_{\frY\rightarrow\frX}, i^{-1}\sM) =  {\mathcal Hom}_{\sA_{\frY}}(\sN,  {\mathcal Hom}_{i^{-1}\sA}( \sA_{\frY\rightarrow\frX},i^{-1}\sM)).$$
\end{proof} 

\begin{cor} Let $ \frZ\stackrel{k}{\rightarrow} \frY\stackrel{i}{\rightarrow}\frX$ be closed immersions of smooth 
formal $\fro$-schemes. Then
$(i\circ k)_\sA^\natural =k_\sA^\natural\circ i_\sA^\natural$ as functors from right $\sA$-modules to
 right $\sA_{\frZ}$-modules.
\end{cor}
\begin{proof} This follows from and \ref{prop-comp2} and \ref{prop-adjoint} by uniqueness of adjoint functors. 
\end{proof}
We come to the main result of this paper. 
Let ${\rm Coh}^r(\sA_{\frY})$ and ${\rm Coh}^{r,\frY}(\sA)$ be the categories of coherent 
 right $\sA_{\frY}$-modules and coherent 
 right $\sA$-modules with support in $\frY$, respectively. 

\begin{thm}\label{thm-BK_A} {\rm (Berthelot-Kashiwara for twisted sheaves)} The functors $i_{\sA,+},i_\sA^\natural$ induces 
mutually inverse equivalences of categories
$$
\xymatrix{
 {\rm Coh}^r(\sA_{\frY}) \;\;\ar@<1ex>[r]^{i_{\sA,+}} &\;\; {\rm Coh}^{r,\frY}(\sA)  \ar@<1ex>[l]^{i_\sA^\natural}_{\simeq}
}.
$$

\end{thm}
\begin{proof} Let $\sN$ be a coherent right $\sA_{\frY}$-module and $\sM$ be a coherent right $\sA$-module supported on $\frY$.
It suffices to see that the counit $i_{\sA,+}\circ i_\sA^\natural (\sM) \rightarrow \sM$ and the unit $ \sN \rightarrow i_\sA^\natural \circ i_{\sA,+}(\sN)$ of the adjunction are isomorphisms. These are local problems and we may therefore suppose that $\sA=\sD^{\dagger}_{\frX,\bbQ}$.
Then $\sA_{\frY\rightarrow\frX}=\sD^{\dagger}_{\frY\rightarrow\frX,\bbQ}$ and
the pair $(i_{\sA,+},i_\sA^\natural)$ becomes the adjoint pair $(i_{r+},i_r^\natural)$ studied in \ref{subsection_right}. Hence \ref{thm_BK_r} implies the assertions. \end{proof}

\section{Applications to crystalline distribution algebras} 

Let $G$ be a connected split reductive group scheme over $ \fro$.

\subsection{The crystalline distribution algebra}
We briefly review the construction of the crystalline distribution algebra of $G$, as introduced in \cite{HS1}. 
Let $I$ be the kernel of the morphism $\fro$-algebras $\varep_G: \fro[G]\rig \fro$ which represents $1\in G$.
Then $I/I^2$ is a free $\fro=\fro[G]/I$-module of finite rank. Let $t_1,\ldots,t_N\in I$ whose classes
modulo $I^2$ form a base of $I/I^2$. The $m$-PD-envelope of $I$ is denoted by $P_{(m)}(G)$. This algebra is a free $\fro$-module
with basis $$\ut^{\{\uk\}}=t_1^{\{k_1\}}\cdots t_N^{\{k_N\}},$$
where $q_i!t_i^{\{k_i\}}=t_i^{k_i}$ with $i=p^mq_i+r$ et $r<p^m$ \cite[1.5]{BerthelotDI}.
The algebra $P_{(m)}(G)$ has a descending filtration by the ideals
$ I^{\{n\}}=\bigoplus_{|\uk|\geq n} \fro \cdot \ut^{\{\uk\}}.$
The quotients $P^n_{(m)}(G):= P_{(m)}(G)/I^{\{n+1\}}$ are generated, as $\fro$-module, by the elements $\ut^{\{k\}}$ where
$|\uk|\leq n$ and there is an isomorphism $P^n_{(m)}(G)\simeq \bigoplus_{|\uk|\leq n} \fro \ut^{\{\uk\}}$ as $\fro$-modules.
There are canonical surjections $pr^{n+1,n}: P^{n+1}_{(m)}(G)\trig P^n_{(m)}(G)$.

\vskip5pt 

We note $$\Lie(G):=\Hom_\fro(I/I^2,\fro).$$ The Lie-algebra $\Lie(G)$ is a free $\fro$-module with
basis $\xi_1,\ldots,\xi_N$ dual to $t_1,\ldots,t_N.$ For $m'\geq m$, the universal property of divided power algebras gives
homomorphismes of filtered algebras
$\psi_{m,m'}\,\colon \, P_{(m')}(G)\rig P_{(m)}(G)$
which induce on quotients homomorphismes of algebras
$\psi_{m,m'}^n \,\colon \,P^n_{(m')}(G)\rig P^n_{(m)}(G).$
 {\it The module of distributions of level $m$ and order $n$} is
$D_n^{(m)}(G):=\Hom_{\fro}(P^n_{(m)}(G),\fro)$
{\it The algebra of distributions of level $m$} is defined to be
$$D^{(m)}(G):=\varinjlim_n D_n^{(m)}(G)$$ where the limit is taken with respect to the maps $\Hom_{\fro}(pr^{n+1,n},\fro)$.

\vskip5pt

For $m'\geq m$, the algebra homomorphisms $\psi_{m,m'}^n$ give dually linear maps $\Phi_{m,m'}^n$ : $D^{(m)}_n(G)\rig D^{(m')}_n(G)$ and finally a morphism of filtered algebras
$\Phi_{m,m'}: D^{(m)}(G)\rig D^{(m')}(G).$ The direct limit ${\rm Dist}(G)=\varinjlim_m D^{(m)}(G)$
equals the classical distribution algebra of the group scheme $G$ \cite[II.\S 4.6.1]{DemazureGabriel}.
Instead of passing to this limit, we let $\GG$ be the completion of $G$ along its special fibre.
 We write $G_i=\Spec\; \fro [G]/\pi^{i+1}$. The morphism $G_{i+1}\hrig G_i$ induces
$D^{(m)}(G_{i+1})\rightarrow D^{(m)}(G_i)$. We put
$$\widehat{D}^{(m)}(\GG):=\varprojlim_{i}D^{(m)}(G_i).$$ If $m'\geq m$, one has the morphisms $\hat{\Phi}_{m,m'}: \widehat{D}^{(m)}(\GG)\rig
\widehat{D}^{(m')}(\GG)$ and one can define the {\it crystalline distribution algebra} of $\GG$ as 
$$
D^{\dagger}(\GG)_{\bbQ}:= \varinjlim_{m}\widehat{D}^{(m)}(\GG) \otimes\Q.$$



\subsection{Twisted sheaves on the flag variety}
We let $B\subset G$ be a Borel subgroup scheme containing a maximal split torus $T$, with unipotent radical $N$. Let $X:=G/B$ be the flag scheme and let $\tilde{X}:=G/N$. The right translation action of $T$ on $\tilde{X}$ makes the projection
$$\xi:\tilde{X}\longrightarrow X$$ a $T$-torsor over $X$.

\vskip5pt 
We exhibit a certain class of twisted sheaves of arithmetic differential operators on the $p$-adic completion $\frX$ of $X$. This construction goes back to the fundamental work of Beilinson-Bernstein \cite{BB81} and Borho-Brylinski \cite{BoBr89}. It was adapted to the setting of arithmetic differential operators by Sarrazola-Alzate \cite{SA1}.\footnote{Note that \cite{SA1} assumes $\fro=\bbZ_p$, however, a large part of the results and constructions are valid in full generality.}

\vskip5pt

Let $\cT$ and $\tilde{\frX}$ be the completion of $T$ and $\tilde{X}$ respectively. Then $\cT$ acts from the right on $\tilde{X}$.
We also write $\xi$ for the projection morphism $\tilde{\frX}\rightarrow \frX$ arising from $\xi$ by completion.
We put $$\tilde{\sD}^{\dagger}_{\frX,\Q}:=(\xi_*(\sD^{\dagger}_{\tilde{\frX},\Q}))^\cT.$$
The right $\cT$-action on $\tilde{\frX}$ leads to a central embedding
$D^{\dagger}(\cT)_\Q\rightarrow \tilde{\sD}^{\dagger}_{\frX,\Q}$.

\vskip5pt 

Now we fix a character of the crystalline distribution algebra $D^{\dagger}(\cT)_\Q$ of $\cT$, i.e. a homomorphism

$$ \lambda:  D^{\dagger}(\cT)_\Q\longrightarrow K.$$

Note that, by restriction, $\lambda$ may be viewed as a weight, i.e. a linear form of the Cartan subalgebra 
$$\frt:=Lie(T)\otimes \Q\subset \frg=Lie(G)\otimes\Q,$$ but not every weight extends to a character of $D^{\dagger}(\cT)_\Q$. 
 If $\lambda\in X(T)$, i.e. $\lambda$ lifts to an algebraic character of $T$, then we denote the corresponding $\cG$-equivariant line bundle on $\frX$ (with $p$ inverted) by $\cO_{\frX, \bbQ,\lambda}$. Finally, we denote the trivial character (which restricts to zero on $Lie(T)$) by $\lambda=0$. 

\vskip5pt 

We define
$$\sD^{\dagger}_{\frX,\Q,\lambda}:=\tilde{\sD}^{\dagger}_{\frX,\Q}
\otimes_{D^{\dagger}(\cT)_\Q, \lambda} K,$$ 
compare \cite[Def. 5.0.1]{SA1}. Since $(\sD^{\dagger}_{\cT,\Q})^{\cT}=D^{\dagger}(\cT)_\Q$ according to \cite[Thm. 4.4.9.2]{HS1}, the ring $\tilde{\sD}^{\dagger}_{\frX,\Q}$ is locally, on an open subset trivializing the torsor, isomorphic to $\sD^{\dagger}_{\frX,\Q}\otimes_K D^{\dagger}(\cT)_\Q$. It follows that $\sD^{\dagger}_{\frX,\Q, \lambda}$ is indeed locally isomorphic to $\sD^{\dagger}_{\frX,\Q}$, i.e. $\sD^{\dagger}_{\frX,\Q, \lambda}$ is a twisted sheaf of arithmetic differential operators on $\frX$ depending on the character $\lambda$.
 If $\lambda\in X(T)$, then there is a natural left action of $\sD^{\dagger}_{\frX,\Q, \lambda}$ 
 on $\cO_{\frX,\bbQ,\lambda}$. For $\lambda=0$ one recovers $\sD^{\dagger}_{\frX,\Q,0} =  \sD^{\dagger}_{\frX,\Q}$. 
 
 \vskip10pt 
 
 We need to determine the opposite ring of $\sD^{\dagger}_{\frX,\Q,\lambda}$ in terms of the weight $\lambda$.
 To do this, for any ring $A$, we denote by $A^{opp}$ its opposite ring, i.e. the same underlying abelian group, but where the order of multiplication reversed. For any $\bbZ_p$-algebra $A$, we denote by $\hat{A}$ its $p$-adic completion and let $\hat{A}_{\bbQ}:=\hat{A}\otimes\bbQ$.
 
 \begin{lemma} There are ring isomorphisms $(\hat{A})^{opp}\simeq \widehat{A^{opp}}$ and $(A^{opp})_{\bbQ}=(A_{\bbQ})^{opp}$.
 \end{lemma}
 \begin{proof} The canonical ring homomorphism $A\rightarrow (\widehat{A^{opp}})^{opp}$ extends to a bijection $\hat{A}\simeq (\widehat{A^{opp}})^{opp}$. Passing to opposite rings gives the first claim. The argument in the second case is similar. 
 \end{proof}

 Let $\rho=\frac{1}{2}\sum_{\Phi^+}\alpha$ be half the sum over the positive roots of $(G,T)$ with respect to $B$.
 
 \begin{prop}\label{opposite_ring2} There is a ring isomorphism $(\sD^{\dagger}_{\frX,\Q,\lambda})^{opp}\simeq \sD^{\dagger}_{\frX,\Q,2\rho-\lambda}$.
  \end{prop}
  \begin{proof} By construction, one has $\sD^{\dagger}_{\frX,\Q,\lambda}=\varinjlim_{m\geq 0}  \hsD^{(m)}_{\frX,\Q,\lambda}$
  with twisted sheaves $\sD^{(m)}_{X,\lambda}$, their $p$-adic completions $\hsD^{(m)}_{\frX,\lambda}$ and their corresponding $\bbQ$-algebras $\hsD^{(m)}_{\frX,\Q,\lambda}$. By the preceding lemma, it suffices to fix $m$ and 
  to show  $(\sD^{(m)}_{X,\lambda})^{opp}\simeq \sD^{(m)}_{\frX,2\rho-\lambda}$. Let $\sA^{(m)}_{X}:=\cO_{X}\otimes_{\fro} D^{(m)}(G)$. According to \cite[3.5.13]{SA1}, there is a surjective morphism of $\cO_X$-rings $$\Phi^{(m)}_{X,\lambda}: \sA^{(m)}_{X}\rightarrow \sD^{(m)}_{X,\lambda}$$ which gives, upon inversion of $p$ and restriction to the generic fibre $\iota: X_\bbQ\hookrightarrow X$, the classical morphism 
  $\Phi_{X_\bbQ,\lambda}:\cO_{X_\bbQ}\otimes_{K} U(\frg)\rightarrow D_{X_\bbQ,\lambda}$.
  Following \cite[4.15.1]{Kashiwara_rims622}, we denote by $\varphi$ the anti-isomorphism of the $\cO_{X_\bbQ}$-ring $\cO_{X_\bbQ}\otimes_{K} U(\frg)$ induced by $x\mapsto -x$ on $\frg$. It restricts to an anti-isomorphism of  $\sA^{(m)}_{X}$, which we also denote by $\varphi$. 
  Let $\cJ^{(m)}:=\ker\Phi^{(m)}_{X,\lambda}$ and $\cJ_\lambda:=\ker\Phi_{X_\bbQ,\lambda}$. Then $\cJ_\lambda^{(m)}= \sA^{(m)}_{X}\cap\iota_*\cJ_\lambda$. Following \cite{BB81}, we denote by $\frb^\circ$ the kernel of the canonical morphism $\alpha: \cO_{X_\bbQ}\otimes_{K} \frg\rightarrow D_{X_\bbQ}$ and by $\lambda^\circ$ the morphism $\frb^\circ\rightarrow\cO_{X_{\bbQ}}$ induced by $\lambda$
  (using that $\frb^\circ/[\frb^\circ,\frb^\circ]\simeq\cO_{X_\bbQ}\otimes_K \frt$). The two-sided ideal $\cJ$ then equals the right-ideal generated by $\ker\lambda^\circ$. By \cite[4.15.1]{Kashiwara_rims622} the anti-isomorphism $\varphi$ maps $\ker\lambda^\circ$ to $\ker(2\rho-\lambda)^\circ$ (note that Kashiwara writes $\tilde{\frg}$ for $\frb^\circ$, cf. \cite[4.3]{Kashiwara_rims622}). It follows $\varphi(\cJ_\lambda)=\cJ_{2\rho-\lambda}$ and this implies $\varphi(\cJ^{(m)}_\lambda)=\cJ^{(m)}_{2\rho-\lambda}$. This proves the proposition.
  \end{proof}

Let in the following $$i: \frY\longrightarrow\frX$$ be a smooth closed formal subscheme defined by some coherent ideal $\cI\subseteq \cO_{\frX}$. 
According to \ref{cor-twist} and \ref{opposite_ring} the sheaf 
$$\sD^{\dagger}_{\frY,\Q,\lambda}:=i^{-1} \big( \sN_{\sD^{\dagger}_{\frX,\Q,2\rho-\lambda} }(\cI \sD^{\dagger}_{\frX,\Q,2\rho-\lambda})/\cI \sD^{\dagger}_{\frX,\Q,2\rho-\lambda}\big)^{opp}$$
is a twisted sheaf of arithmetic differential operators on $\frY$. If $\lambda\in X(T)$, we let 
$$\cO_{\frY,\bbQ,\lambda}:=i^{*}\cO_{\frX,\bbQ,\lambda}=i^{-1} ( \cO_{\frX,\bbQ,\lambda}/ \cI  \cO_{\frX,\bbQ,\lambda}).$$ It is a line bundle on $\frY$.
\begin{prop}\label{simple} Let $\lambda\in X(T)$.\vskip5pt
(i) The line bundle $\cO_{\frY,\bbQ,\lambda}$ is naturally a left $\sD^{\dagger}_{\frY,\Q,\lambda}$-module. 
\vskip5pt
(ii) Assume that the special fibre and the rigid-analytic generic fibre of $\frY$ are connected. Then $\cO_{\frY,\bbQ,\lambda}$ is a simple left $\sD^{\dagger}_{\frY,\Q,\lambda}$-module. 
\end{prop}
\begin{proof} We recall that $\sD^{\dagger}_{\frX,\Q,\lambda}$ naturally acts from the left on $\cO_{\frX,\lambda,\bbQ}$. 
Let $f\in\cO_{\frX,\lambda, \bbQ}$ and $P\in \sD^{\dagger}_{\frX,\Q,2\rho-\lambda}$ be local sections.
Identifying the $\cO_{\frX,\bbQ}$-ring $\sD^{\dagger}_{\frX,\Q,2\rho-\lambda}$ with $(\sD^{\dagger}_{\frX,\Q,\lambda})^{opp}$, 
there is a well-defined local section $P(f)\in\cO_{\frX,\bbQ,\lambda}$. Whether or not the subset 
$\sN_{\sD^{\dagger}_{\frX,\Q,2\rho-\lambda} }(\cI \sD^{\dagger}_{\frX,\Q,2\rho-\lambda})$ stabilizes the 
submodule $\cI  \cO_{\frX,\bbQ,\lambda}\subset   \cO_{\frX,\bbQ,\lambda}$ is a local question. We may hence fix a local $\cO_{\frX,\bbQ}$-linear isomorphism between  $\cO_{\frX,\bbQ,\lambda}$
and $\cO_{\frX,\bbQ}$. If then $P\in \sN_{\sD^{\dagger}_{\frX,\Q,2\rho-\lambda} }(\cI \sD^{\dagger}_{\frX,\Q,2\rho-\lambda})$ and $f\in\cI_\bbQ$, then there is $Q\in \sD_{\frX,\bbQ,2\rho-\lambda }^{\dagger}$ and $h\in \cI_\bbQ$ such that $Pf=hQ.$ It follows $$P(f)=Pf(1)=hQ(1)\in h\cO_{\frX,\bbQ}\subseteq \cI_\bbQ.$$ In this way, the right module structure of $\cO_{\frX,\bbQ,\lambda}$ over $\sD^{\dagger}_{\frX,\Q,2\rho-\lambda}=(\sD^{\dagger}_{\frX,\Q,\lambda})^{opp}$ makes

$\cO_{\frX,\bbQ,\lambda}/ \cI_{\bbQ}\cO_{\frX,\bbQ,\lambda}$ a right module over the ring
$$\sN_{\sD^{\dagger}_{\frX,\Q,2\rho-\lambda} }(\cI \sD^{\dagger}_{\frX,\Q,2\rho-\lambda})/\cI \sD^{\dagger}_{\frX,\Q,2\rho-\lambda}.$$
This implies that $\cO_{\frY,\bbQ,\lambda}=i^{-1} (\cO_{\frX,\bbQ,\lambda}/ \cI_{\bbQ}\cO_{\frX,\bbQ,\lambda})$ is a left module over $\sD^{\dagger}_{\frY,\Q,\lambda}$, as claimed.

For (ii) we assume that the special fibre $\frY_s$ and the rigid-analytic generic fibre $\frY_K$ of $\frY$ are connected.
Let $\cJ\subseteq \cO_{\frY,\bbQ,\lambda}$ be a submodule. Then $\cJ$ is a coherent $\cO_{\frY,\bbQ}$-module,
We show that the intersection $Supp(\cJ)\cap Supp(\cO_{\frY,\bbQ,\lambda}/\cJ)$ is empty. Since 
$Supp(\cJ)$ and $Supp(\cO_{\frY,\bbQ,\lambda}/\cJ)$ are closed subsets and their union equals $\frY_s$, the connectedness of $\frY_s$ implies then that one of them is empty, thus either $\cJ=0$ or $\cJ=\cO_{\frY,\bbQ,\lambda}$. 
Let us assume for a contradiction that $$y\in Supp(\cJ)\cap Supp(\cO_{\frY,\bbQ,\lambda}/\cJ).$$ Choose an open affine rig-connected $\frU\subset\frY$ containing $y$. Making $\frU$ smaller if necessary, we may assume that the restriction of $\sD^{\dagger}_{\frY,\Q,\lambda}$ to $\frU$ is isomorphic to $\sD^{\dagger}_{\frU,\Q}$. But then the sheaf $\cJ |_\frU$ is a nonzero and proper $\sD^{\dagger}_{\frU,\Q}$-submodule of $\cO_{\frU,\bbQ}$. This is in contradiction to the fact that $\cO_{\frU,\bbQ}$ is a simple $\sD^{\dagger}_{\frU,\Q}$-module, cf. 
\cite[Prop. 2.3.6]{HS4}.
\end{proof}

\subsection{Geometric construction of simple modules}
We keep the notation. In particular, $ \lambda:  D^{\dagger}(\cT)_\Q\rightarrow K$ is a character giving rise to the twisted sheaf $\sD^{\dagger}_{\frX,\bbQ,\lambda}$ on the flag variety $\frX$. Let $i: \frY\rightarrow\frX$ be a smooth closed formal subscheme with twisted sheaf $\sD^{\dagger}_{\frY,\bbQ,\lambda}$.

\vskip5pt 

Let $\theta: Z(\frg)\rightarrow K$ be a character of the center $Z(\frg)$ of $U(\frg)$, which corresponds to the weight of $\frt$ induced by $\lambda$ under the classical Harish-Chandra homomorphism. We let
 $$D^{\dagger}(\GG)_{\bbQ,\theta}:= D^{\dagger}(\GG)_{\bbQ}\otimes_{Z(\frg),\theta}K$$
 be the corresponding central reduction. We recall the localization theorem for left $D^{\dagger}(\GG)_{\bbQ,\theta}$-modules. 

\begin{thm} \label{localizationthm}
{\rm (a)} Suppose that $\lambda+\rho$ is dominant and regular (as a weight of $\frt$). The global section functor induces an equivalence of categories between coherent left $\sD^{\dagger}_{\frX,\bbQ,\lambda}$-modules and coherent left $H^0(\frX,\sD^{\dagger}_{\frX,\bbQ,\lambda})$-modules. 
\vskip5pt

{\rm (b)} The $\cG$-action on $\frX$ induces an algebra isomorphism
$$D^{\dagger}(\GG)_{\bbQ,\theta}\car H^0(\frX, \sD^{\dagger}_{\frX,\bbQ,\lambda}).$$

\end{thm}
\begin{proof} If $\lambda$ lifts to an algebraic character of $T$, then this summarizes the main results of \cite{HS2} and \cite{NootHuyghe09}. The case of a general character is the main result of \cite{SA1}.
\end{proof}

We come to the main application of our results. 

\begin{thm} \vskip5pt 
(i) There is an equivalence of categories 
$$
i_{+,\lambda}: {\rm Coh}(\sD^{\dagger}_{\frY,\Q,\lambda}) \stackrel{\simeq}{\longrightarrow}  {\rm Coh}^{\frY}(\sD^{\dagger}_{\frX,\Q,\lambda})
$$
preserving simple objects on both sides. If $\lambda\in X(T)$, let $\cB_{\frY | \frX,\lambda}:=i_{+,\lambda} \cO_{\frY,\bbQ,\lambda}$. 
\vskip5pt
(ii) Let $\lambda+\rho$ be dominant and regular and let $\frY$ have connected special and generic fibre. If $\lambda\in X(T)$, then
$H^0(\frX,\cB_{\frY | \frX,\lambda})$ is a simple $D^{\dagger}(\GG)_{\bbQ,\theta}$-module. 
If two such modules $H^0(\frX,\cB_{\frY | \frX,\lambda})$ and $H^0(\frX,\cB_{\frY' | \frX,\lambda})$ are isomorphic, then $\frY_s=\frY'_s$. 
\end{thm}
\begin{proof}
The point (i) follows from \ref{opposite_ring2} and the Berthelot-Kashiwara theorem for right modules over 
$(\sD^{\dagger}_{\frX,\Q,\lambda})^{opp}= \sD^{\dagger}_{\frX,\Q,2\rho-\lambda}$, cf. \ref{thm-BK_A}. The point (ii) follows from 
\ref{simple} together with the localisation theorem. 
\end{proof}
Remark: In the setting of (ii), an isomorphism $H^0(\frX,\cB_{\frY | \frX,\lambda})\simeq H^0(\frX,\cB_{\frY' | \frX,\lambda})$ does {\it not} in general imply an equality of $\frY$ and $\frY'$ as formal subschemes of $\frX$, even in the case $\lambda = 0$.

\bibliographystyle{plain}
\bibliography{mybib}

\end{document}